\documentclass[preprint,12pt]{elsarticle}

\usepackage{amssymb}

\usepackage{amsmath}
\usepackage{amssymb}
\usepackage{subfigure, graphicx}

\usepackage{bbm}
\usepackage{a4}
\usepackage{enumerate}
\usepackage{epsf}
\usepackage{graphicx}
\usepackage{hyperref}

\newcommand{\Partial}[3]{\ensuremath{\frac{\partial^{#1}{#2}}{\partial{#3}^{#1}}}}
\renewcommand{\Partial}[3][ ]{\ensuremath{\partial^{#1}_{#3} {#2}}}

\newcommand{\bfvec}[1]{\ensuremath{{\mathbf #1}}}

\newcommand{\grad}{\ensuremath{ \nabla }}

\newcommand{\Div}[1]{\ensuremath{ \nabla \cdot \left( #1 \right)}}

\newcommand{\clos}[1]{\ensuremath{\overline{#1}}}
\newcommand{\dx}{\ensuremath{{\rm dx}}}
\newcommand{\dS}{\ensuremath{{\rm dS}}}

\newcommand{\Gammat}{\ensuremath{{\Gamma\left(t\right)}}}

\newcommand{\R}[1]{\ensuremath{\mathbbm{R}^{#1}}}

\newcommand{\real}[2]{\ensuremath{{#1}\cdot 10^{#2}}}
\newcommand{\Uh}{\ensuremath{U^h}}
\newcommand{\uh}{\ensuremath{u^h}}
\newcommand{\Vh}{\ensuremath{V^h}}
\newcommand{\gradh}{\ensuremath{\nabla^h}}

\newtheorem{theorem}{Theorem}[section]
\newtheorem{definition}[theorem]{Definition}
\newtheorem{lemma}[theorem]{Lemma}
\newproof{proof}{Proof}

\journal{???}

\begin{document}

\begin{frontmatter}



\title{Numerical solution for the anisotropic Willmore flow of graphs}
\ead{tomas.oberhuber@fjfi.cvut.cz}
\ead[url]{http://geraldine.fjfi.cvut.cz/~oberhuber}


\author{Tom\'{a}\v{s} Oberhuber}
\address{Department of Mathematics, Faculty of Nuclear Sciences and Physical Engineering, Czech Technical University in Prague, Trojanova 13, Praha 2, 120 00, Czech Republic}

\begin{abstract}
The Willmore flow is well known problem from the differential geometry. It minimizes the Willmore functional defined as integral of the mean-curvature square over given manifold. For the graph formulation, we derive modification of the Willmore flow with anisotropic mean curvature. We define the weak solution and we prove an energy equality. We approximate the solution numerically by the complementary finite volume method. To show the stability, we re-formulate the resulting scheme in terms of the finite difference method. By using simple framework of FDM we show discrete version of the energy equality. The time discretization is done by the method of lines and the resulting system of ODEs is solved by the Runge-Kutta-Merson solver with adaptive integration step. We also show experimental order of convergence as well as results of the numerical experiments, both for several different anisotropies.
\end{abstract}


\end{frontmatter}

\section{Introduction} This article extends the isotropic Willmore flow defined by T. J. Willmore in \cite{Willmore-2002}. It is a minimizer of the Willmore functional defined as
$$
\mathcal{W}\left(\Gamma\right) = \int_\Gamma H^2 dS,
$$
where $\Gamma$ is given manifold smooth enough so that the mean curvature $H$ can be evaluated almost everywhere on $\Gamma$. Prescribing the following normal velocity
$$
V = \triangle_{\Gamma} H + H^3 - 2HK\ {\rm on}\ \Gamma\left( t \right),
$$
we generate class of manifolds $\Gammat$ so that $\Gammat$ minimizes the Willmore functional as $t$ goes to infinity. One recognizes several formulation of the Willmore depending of the form in which $\Gammat$ is expressed - {\it the graph formulation} \cite{Dec_Dzi_EEFTWFOG,Oberhuber-2006,XuShu-2009}, {\it the level-set formulation} \cite{Dro_Rum_ALSFFWF,Benes-Mikula-Oberhuber-Sevcovic-2007-1} or {\it the parametric formulation} \cite{Dzi_Kuw_Sch_EOECIRNEAC,Benes-Mikula-Oberhuber-Sevcovic-2007-1}. Also approximation by {\it the phase-field model} exists \cite{Du-Liu-Ryham-Wang-2005}. In this article we will study only the graph formulation. Applications of the Willmore flow can be found in biology \cite{Canham-1970} or in image processing in image inpainting \cite{BertalmioCasellesHaroSapiro-2006}. Especially in the later one, the anisotropic model can be useful.

\section{Problem formulation}
We assume having a manifold $\Gamma_0$ described as a graph of function $u_0$ of two variables:
\begin{equation}
\label{def:graph-formulation}
\Gamma_0 = \left\{ \left[ {\bf x}, u_0\left( {\bf x} \right) \right] \mid {\bf x} \in \Omega \subset \mathbbm{R}^2 \right\},
\end{equation}
where $\Omega \equiv \left( 0, L_1 \right) \times \left( 0, L_2 \right)$ is an open rectangle and we will denote $\partial \Omega$ its boundary.
%
%
We assume having convex function $\gamma: \mathbbm{R}^{n+1}\setminus \{0\} \rightarrow \mathbbm{R}_0^+$, $\gamma = \gamma(p_1, \cdots p_n, -1)$ which is positive 1-homogeneous i.e. $\gamma\left(\lambda \bfvec p\right) = \left| \lambda \right| \gamma\left(\bfvec p\right)$. We will call $\gamma$ surface energy density and denote
$$\nabla_{p} \gamma = \left( \Partial{\gamma}{p_1}, \cdots, \Partial{\gamma}{p_n} \right),$$
where $\Partial{\gamma}{p_i}$ stands for partial derivatives of $\gamma$ w.r.t. variable $p_i$. We define anisotropic mean curvature of a manifold $\Gamma_0$ smooth enough induced by energy density function $\gamma$ as
\begin{equation}
\label{aniso-mean-curvature}
H_\gamma\left(u_0\right) = \nabla \cdot  \left( \nabla_{p} \gamma \left( \grad u_0, -1 \right) \right).
\end{equation}
and the anisotropic Willmore functional as 
\begin{equation}
\label{willmore-functional-graphs}
\mathcal{W}_\gamma\left(u_0\right) = \frac1{2} \int_\Omega H_\gamma^2\left(u_0\right) Q\left(u_0\right) \dx,
\end{equation}
for $Q\left(u\right)=\sqrt{1+\left|\grad u \right|^2}$. We aim to find a function $u^\ast$ minimizing (\ref{willmore-functional-graphs}). The Euler-Lagrange equation for this functional takes the following form
\begin{equation}
\label{willmore-euler-lagrange-aniso}
\Div{ \mathbbm{E}_\gamma\left(u^\ast\right) \grad w_\gamma\left(u^\ast\right) - \frac{1}{2}\frac{w^2_\gamma\left(u^\ast\right)}{Q^3\left(u^\ast\right)} \grad u^\ast } = 0,
\end{equation}
where we denoted
\begin{eqnarray*}
w_\gamma\left(u\right) &:=& Q\left(u\right) H_\gamma\left(u\right), \\
\mathbbm{E}_\gamma\left(u\right)  &:=& \Partial{}{p_i} \Partial{}{p_j} \gamma\left(\grad u,-1\right) =
\left( \grad_{\bfvec p} \otimes \grad_{\bfvec p} \right) \gamma\left(\grad u, -1\right).
\end{eqnarray*}
In the rest of the text we will not emphasize explicitly the dependence of $Q, H_\gamma, w_\gamma$ and $\mathbbm{E}_\gamma$ on $u$. Multiplying (\ref{willmore-euler-lagrange-aniso}) by the test function $\varphi \in C^\infty\left(\Omega\right)$ and integrating over $\Omega$ we get
\begin{eqnarray*}
&& 0 = \int_\Omega \Div{\mathbbm{E}_\gamma \grad w_\gamma - \frac{1}{2}\frac{w_\gamma^2}{Q^3} \grad u} \varphi \dx = \\
&& -\int_{\partial \Omega} \mathbbm{E}_\gamma \grad w_\gamma \cdot \nu \varphi  + \frac{1}{2} \frac{w_\gamma^2}{Q^3} \grad u \cdot \nu \varphi \dS +
\int_\Omega \Div{\mathbbm{E}_\gamma \grad w_\gamma} \varphi - \frac{1}{2} \Div{\frac{w_\gamma^2}{Q^3} \grad u} \varphi \dx.
\end{eqnarray*}
The boundary integral over $\partial \Omega$ vanishes if we set
\begin{eqnarray}
\label{def:neumann-bc-1}
\partial_\nu u &=& 0\ {\rm on}\ \partial \Omega, \\
\label{def:neumann-bc-2}
\mathbbm{E}_\gamma \grad w_\gamma \cdot \nu &=& 0\ {\rm on}\ \partial \Omega.
\end{eqnarray}
These relations define the Neumann boundary conditions. The Dirichlet boundary conditions 
\begin{eqnarray}
\label{def:dirichlet-bc-1}
u &=& g_1\ {\rm on}\ \partial \Omega, \\
\label{def:dirichlet-bc-2}
w_\gamma &=& g_2\ {\rm on}\ \partial \Omega,
\end{eqnarray}
may be obtained in the same way but taking $\varphi \in C_0^\infty \left(\Omega\right)$. In practice, the $L_2$-gradient flow of (\ref{willmore-functional-graphs}) is solve rather than (\ref{willmore-euler-lagrange-aniso}) -- see \cite{Dec_Dzi_EEFTWFOG,Dro_Rum_ALSFFWF}. The following definition involves parabolic partial differential equation with unknown function $u$ which is now also dependent on artificial time parameter $t$ i.e. $u=u\left(t,\bfvec x\right)$. In the same sense as in (\ref{def:graph-formulation}), a moving manifold $\Gammat$ is obtained.
\begin{definition}
\label{def:aniso-willmore-fl-graphs}
Let $\Omega$ be a domain in $\R{2}$. The {\bf anisotropic Willmore flow of graphs with the Dirichlet boundary conditions} and the initial condition $u_0$ is a fourth order parabolic problem given by
\begin{eqnarray}
\label{anisotropic-willmore-graphs-ut}
\partial_t u &=&-Q\Div{ \mathbbm{E}_\gamma \grad w_\gamma - \frac{1}{2}\frac{w^2_\gamma}{Q^3} \grad u }\quad {\rm on}\ \left(0,T\right) \times \Omega,\\
\label{anisotropic-willmore-graphs-w}
w_\gamma &=& QH_\gamma \quad {\rm on}\ \left(0,T\right) \times \Omega,\\
\label{anisotropic-willmore-graphs-ini}
u \mid_{t=0} &=& u_0 \quad {\rm on}\ \Omega, \\
\label{aniso-willmore-graphs-dirichlet-bc}
u &=& g_1,\ w_\gamma = g_2  \quad {\rm on}\ \partial \Omega.
\end{eqnarray}
The {\bf anisotropic Willmore flow of graphs with the Neumann boundary conditions} and the initial condition $u_0$ is a fourth order parabolic problem given by (\ref{anisotropic-willmore-graphs-ut})--(\ref{anisotropic-willmore-graphs-ini}) and 
\begin{equation}
\label{aniso-willmore-graphs-neumann-bc}
\Partial{u}{\nu} = 0,\ \mathbbm{E}_\gamma \grad w_\gamma \cdot \nu = 0 \quad {\rm on}\ \partial \Omega.
\end{equation}
\end{definition}
We also define the weak formulation:
\begin{definition}
Let $\Omega$ be a domain in $\R{2}$. The {\bf weak solution of anisotropic Willmore flow of graphs with the Dirichlet boundary conditions}
\begin{eqnarray*}
u &=& g_1 \quad {\rm on} \ \partial \Omega,\\
w_\gamma &=& g_2 \quad {\rm on} \ \partial \Omega,
\end{eqnarray*}
is a couple $u,w_\gamma : \left( 0, T \right) \rightarrow  H^1_0\left(\Omega\right)$ which for each test function $\varphi, \xi \in H^1_0\left(\Omega \right)$ and a.e in $\left( 0, T \right)$ satisfies, 
\begin{eqnarray}
\label{willmore-graphs-weak-solution-ut}
\int_{\Omega} \frac{u_t}{Q} \varphi \dx &=& \int_{\Omega} \left( \mathbbm{E}_\gamma \grad w_\gamma \right) \cdot \grad \varphi - \frac{1}{2}\frac{w_\gamma^2}{Q^3} \grad u \cdot \grad \varphi \dx\ {\rm a.e.\ in}\ \left( 0, T \right) \\
\label{willmore-graphs-weak-solution-w}
\int_{\Omega} \frac{w_\gamma}{Q} \xi \dx &=& - \int_{\Omega} \grad_{\bfvec p} \gamma \cdot \grad \xi \dx.
\end{eqnarray}
with the initial condition
\begin{equation}
u \mid_{t=0} = u_0.
\label{willmore-graphs-weak-solution-initial-cond}
\end{equation}
The {\bf weak solution of anisotropic Willmore flow of graphs with homogeneous Neumann boundary conditions}
\begin{eqnarray*}
\Partial{u}{\nu} &=& 0 \quad {\rm on} \ \partial \Omega,\\
\mathbbm{E}_\gamma \grad w \cdot \nu &=& 0 \quad {\rm on} \ \partial \Omega,
\end{eqnarray*}
is a couple $u,w: \left( 0, T \right) \rightarrow H^1\left(\Omega\right)$ which for each test function $\varphi, \xi \in H^{1}\left(\Omega \right)$ and a.e. in $\left( 0, T \right)$ satisfies (\ref{willmore-graphs-weak-solution-ut})-(\ref{willmore-graphs-weak-solution-w}) and the initial condition (\ref{willmore-graphs-weak-solution-initial-cond}).
\end{definition}

For the proof of the numerical stability we will need the following theorem:
\begin{theorem}
\label{awg-energy-equality}
For the solution $u,w_\gamma$ of (\ref{willmore-graphs-weak-solution-ut})-(\ref{willmore-graphs-weak-solution-initial-cond}) with the zero Dirichlet boundary conditions the following energy equality holds:
\begin{equation}
\int_\Omega \frac{\left(\partial_t u \right)^2}{Q} \dx + \frac{1}{2} \frac{{\rm d}}{{\rm d}t} \int_\Omega H^2_\gamma Q \dx = 0.
\end{equation}
\end{theorem}
\begin{proof}
We differentiate (\ref{willmore-graphs-weak-solution-w}) with respect to $t$
\begin{equation}
\int_\Omega \frac{\partial_t w_{\gamma} \xi}{Q} \dx - 
\int_\Omega \frac{w_\gamma \xi \partial_t Q}{Q^2} \dx + 
\int_\Omega \mathbbm{E}_\gamma \grad \partial_t u \cdot \grad \xi = 0 \quad {\rm for\ all\ } \xi \in H_0^1\left(\Omega\right)
\label{awg-aux-9}
\end{equation}
which follows from
\begin{eqnarray*}
\frac{{\rm d}}{{\rm d}t} \grad_{\bfvec p} \gamma \left(\grad u, - 1 \right) \cdot \grad \xi &=&
\frac{{\rm d}}{{\rm d}t} \sum_{i=1}^n \Partial{}{p_i}\gamma \left(\grad u, - 1 \right) \Partial{\xi}{p_i} \\
&=&\sum_{i,j=1}^n \Partial{}{p_i} \Partial{}{p_j} \gamma \left(\grad u, -1\right) \Partial{}{t} \Partial{u}{x_j} \Partial{\xi}{p_i} \\
&=& \mathbbm{E}_\gamma \grad \partial_t u \cdot \grad \xi.
\end{eqnarray*}
Substituting $\varphi = \Partial{u}{t}$ in (\ref{willmore-graphs-weak-solution-ut}) and $\xi =w_\gamma$ in (\ref{awg-aux-9}) we have
\begin{eqnarray}
\label{wg-aux-6}
\int_{\Omega} \frac{\left(\partial_t u\right)^2}{Q} \dx - \int_{\Omega} \left( \mathbbm{E}_\gamma \grad w_\gamma \right) \cdot \grad \partial_t u \dx + \int_\Omega \frac{1}{2}\frac{w^2_\gamma}{Q^3} \grad u \cdot \grad \partial_t u \dx &=& 0, \\
\label{wg-aux-7}
\int_\Omega \frac{\partial_t w_\gamma w_\gamma}{Q} \dx - 
\int_\Omega \frac{w^2_\gamma \partial_t Q}{Q^2} \dx + 
\int_\Omega \left( \mathbbm{E}_\gamma \grad \partial_t u \right) \cdot \grad w_\gamma &=& 0 
\end{eqnarray}
The sum of (\ref{wg-aux-6}) and (\ref{wg-aux-7}) gives
$$
\int_{\Omega} \frac{\left( \partial_t u \right)^2}{Q} + 
\frac{\partial_t w_\gamma w_\gamma}{Q} - \frac{w^2_\gamma \partial_t Q}{Q^2}  + 
\frac{1}{2}\frac{w^2_\gamma}{Q^3} \grad u \cdot \grad \partial_t u \dx = 0.
\label{wg-aux-8}
$$
Since $\grad \Partial{u}{t} \cdot \grad u = \Partial{Q}{t}Q$, (\ref{wg-aux-8}) turns to
$$
\int_{\Omega} \frac{\left(\partial_t u \right)^2}{Q} + 
\frac{\partial_t w_\gamma w_\gamma}{Q} - \frac{1}{2}\frac{w^2_\gamma \partial_t Q}{Q^2} \dx = 0,
$$
which is indeed what we wanted to show because
$$
\frac{1}{2} \frac{{\rm d}}{{\rm d}t} H^2_\gamma Q = 
\frac{1}{2} \frac{{\rm d}}{{\rm d}t} \frac{w^2_\gamma}{Q} = 
\frac{\partial_t w_\gamma w_\gamma}{Q} - \frac{1}{2}\frac{w^2_\gamma \partial_t Q}{Q^2}.
$$
\end{proof}

We will demonstrate the anisotropic Willmore flow of graphs on the following (an)isotropies:
\begin{eqnarray}
\label{gamma-iso-gr}
\gamma_{iso} \left( \bfvec p \right) &:=& \sqrt{ 1 + \left| \bfvec p \right|^2}, \\ 
\label{def:quadratic-form-anisotropy}
\gamma_{\mathbbm G}\left( {\bf p}, -1 \right) &:=& \sqrt{1 + {\bf p}^T \mathbbm{G} {\bf p}}, \\
\label{def:gamma2}
\gamma_{abs}\left( \bfvec P \right) &:=& \sum_{i=1}^3 \sqrt{ P_i^2 + \epsilon_{abs} \sum_{j=1}^3 P_j^2},
\end{eqnarray}
where $\mathbbm G$ is symmetric positive definite matrix $\mathbbm G \in \R{2,2}$ and we use notation $\bfvec P = \left(\bfvec p, -1\right)$. In fact, $\gamma_{iso}$ represents the isotropic problem. The Wulf shapes $W$ of given anisotropies defined by
\begin{equation}
W = \bigcap_{\left| \bfvec q \right|=1} \left\{\bfvec x \in \R{n} \mid (\bfvec x, \bfvec q) \leq \gamma\left(\bfvec q\right) \right\},
\end{equation}
are depicted on the Figure \ref{fig:wulff-shapes}.
\begin{figure}
\center{
\includegraphics[width=3cm]{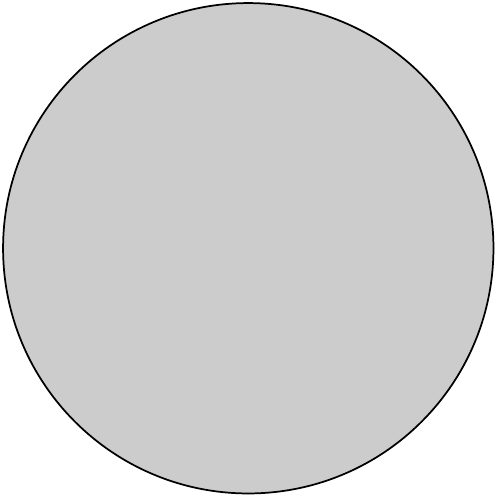}
\hspace{1cm}
\includegraphics[width=3cm]{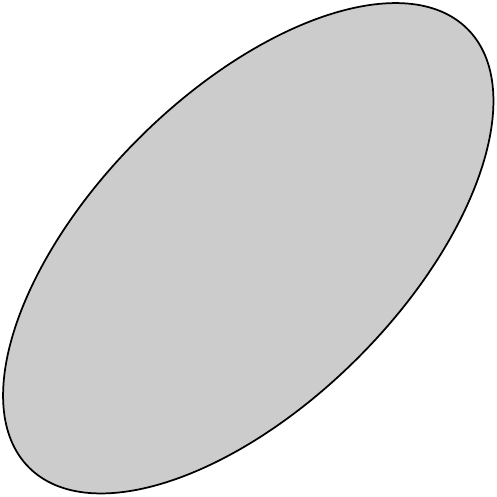}
\hspace{1cm}
\includegraphics[width=3cm]{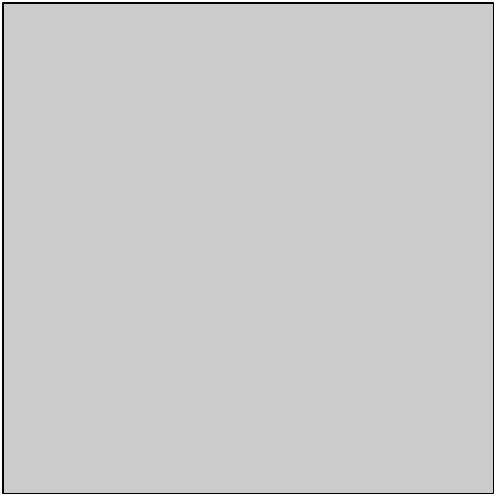}
}
\caption{The Wulf shapes of $\gamma_{iso}$, $\gamma_{\mathbbm{G}}$ and $\gamma_{abs}$ from left to right.}
\label{fig:wulff-shapes}
\end{figure}
Simple calculations show that
\begin{eqnarray}
H_{\gamma_{iso}} &=& \Div{\frac{\grad u}{\sqrt{ 1 + \left| \grad u \right|^2}}}, \quad
\mathbbm{E}_{\gamma_{iso}} = \frac{1}{Q}\left( \mathbbm{I} - \frac{\grad u}{Q} \otimes \frac{\grad u}{Q}\right), \\
H_{\gamma_{\mathbbm{G}}} &=& \Div{\frac{\mathbbm{G}\grad u}{\gamma_{\mathbbm{G}}}}, \quad
\mathbbm{E}_{\gamma_{\mathbbm{G}}} = \frac{1}{\gamma_{\mathbbm{G}}} \left(\mathbbm{G} - \frac{\mathbbm{G} \grad u}{\gamma_{\mathbbm{G}}} \otimes \frac{\mathbbm{G}\grad u}{\gamma_{\mathbbm{G}}}\right).
\end{eqnarray}
It is difficult to express $H_{\gamma_{abs}}$ and $\mathbbm{E}_{\gamma_{abs}}$ in some compact form and so we only show partial derivatives of $\gamma_{abs}$ with respect to $p_i$ and $p_j$ for $i,j=1,2$.
\begin{eqnarray*}
\gamma_{abs,p_i} &=& \sum_{j=1}^3 \frac{\epsilon_{abs} p_i}{\sqrt{P_j^2 + \epsilon_{abs} \sum_{k=1}^3 P_k^2}} + \frac{p_i}{\sqrt{p_i^2 + \epsilon_{abs} \sum_{j=1}^3 P_j^2}} \quad {\rm for}\quad  i = 1,2, \\
\gamma_{abs,p_ip_i} &=& \sum_{j=1}^3 \left( \frac{\epsilon_{abs}}{\sqrt{P_j^2 + \epsilon_{abs}\sum_{k=1}^3 P_k^2}} - 
                                        \frac{\epsilon_{abs}^2 p_i^2}
                                             {\left(P_j^2 + \epsilon_{abs}\sum_{k=1}^3 P_k^2 \right)^\frac{3}{2}} \right) \nonumber \\
                &+&              \frac{1}{\sqrt{p_i^2 + \epsilon_{abs} \sum_{j=1}^3P_j^2}} -
                                \frac{p_i^2}{\left( p_i^2 + \epsilon_{abs} \sum_{j=1}^3 P_j^2\right)^\frac{3}{2}} \quad {\rm for} \quad i = 1,2, \\
\gamma_{abs,p_ip_j} &=& - \sum_{k=1}^3\frac{\epsilon_{abs}^2 p_ip_j}{\left(P_k^2 + \epsilon_{abs} \sum_{l=1}^3 P_l^2\right)^\frac{3}{2}}
                    - \sum_{k=1}^2\frac{\epsilon_{abs} p_ip_j}{\left(P_k^2 + \epsilon_{abs} \sum_{l=1}^3 P_l^2\right)^\frac{3}{2}}.
\end{eqnarray*}

\section{Space discretization}

\subsection{The complementary-finite volume method}

In \cite{Oberhuber-2009} we applied the complementary-finite volume method for the space discretization of the anisotropic surface-diffusion flow. In the same manner we discretize even the problem from the Definition \ref{def:aniso-willmore-fl-graphs}. Let $h_1, h_2$ be space steps such that $h_1 = \frac{L_1}{N_1}$ and $h_2=\frac{L_2}{N_2}$ for some $N_1,N_2 \in \mathbbm{N}^+$. We define a numerical grid, its closure and its boundary as 
\begin{eqnarray}
\label{def:numerical-grid-interior}
\omega_h &=& \left\{ (ih_1, jh_2) \mid i = 1 \cdots N_1 - 1, j = 1 \cdots N_2 - 1 \right\}, \\
\label{def:numerical-grid-closure}
\clos{\omega}_h &=& \left\{ (ih_1, jh_2) \mid i = 0 \cdots N_1, j = 0 \cdots N_2 \right\}, \\
\label{def:numerical-grid-boundary}
\partial \omega_h &=& \clos{\omega_h} \setminus \omega_h.
\end{eqnarray}
For $u \in C\left(\clos{\Omega}\right)$ we define its piecewise constant approximation on $\clos{\omega}$ as a grid function $u^h$ defined as $u^h\left(ih_1,jh_2\right) := u^h_{ij} := u\left(ih_1,jh_2\right)$. We also define a dual mesh $V_h$ as (see the Figure \ref{fig:dual-mesh} a))
\begin{eqnarray}
V_h &\equiv& \left\{ v_{ij} =  \left[ \left(i-\frac{1}{2}\right) h_1, \left(i+\frac{1}{2}\right) h_1 \right] \times
                             \left[ \left(j-\frac{1}{2}\right) h_2, \left(j+\frac{1}{2}\right) h_2 \right] \mid  \right. \nonumber \\
                             &&  i = 1 \cdots N_1-1, j = 1 \cdots N_2-1 \bigg\}.
\label{dual-fvm-mesh}
\end{eqnarray}
For $0 < i < N_1$, $0 < j < N_2$, $i$ and $j$ fixed, consider a finite volume $v_{ij}$ of the dual mesh $V_h$, denote its interior as $\Omega_{ij}$, its boundary as $\Gamma_{ij}$ and let $\mu\left(\Omega_{ij}\right)$ be the volume of $\Omega_{ij}$.  We also denote all the neighboring volumes of the volume $v_{ij}$ as $\mathcal{N}_{ij}$. For all finite volumes $v_{ij}$ of the dual mesh $V_h$, the boundary $\Gamma_{ij}$ consists of four linear segments. We denote them as $\Gamma_{ij,\bar{i}\bar{j}}$. It means that $\Gamma_{ij,\bar{i}\bar{j}}$ is a boundary of the finite volume $v_{ij}$ between nodes $(i,j)$ and $(\bar{i},\bar{j})$. By $l_{ij,\bar{i}\bar{j}}$ we denote the length of this part of $\Gamma_{ij}$.
\begin{figure}
\center{
\includegraphics[width=6cm]{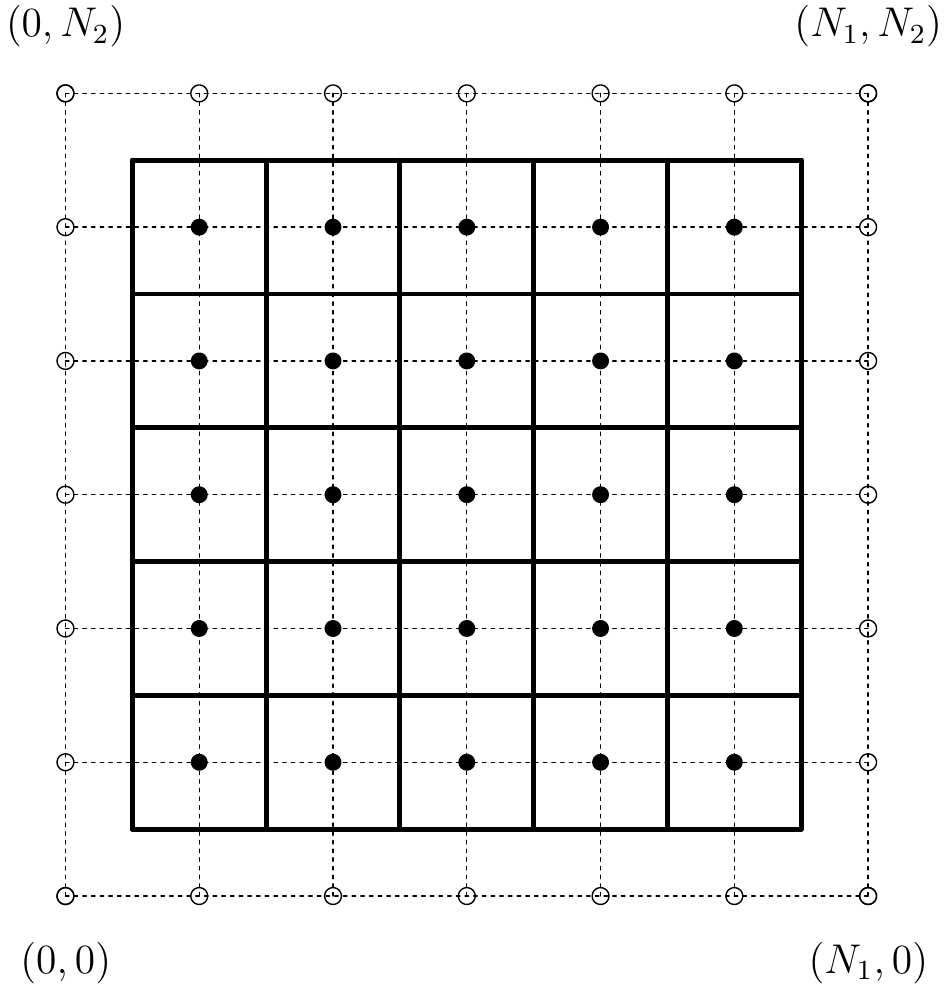}
\includegraphics[width=6cm]{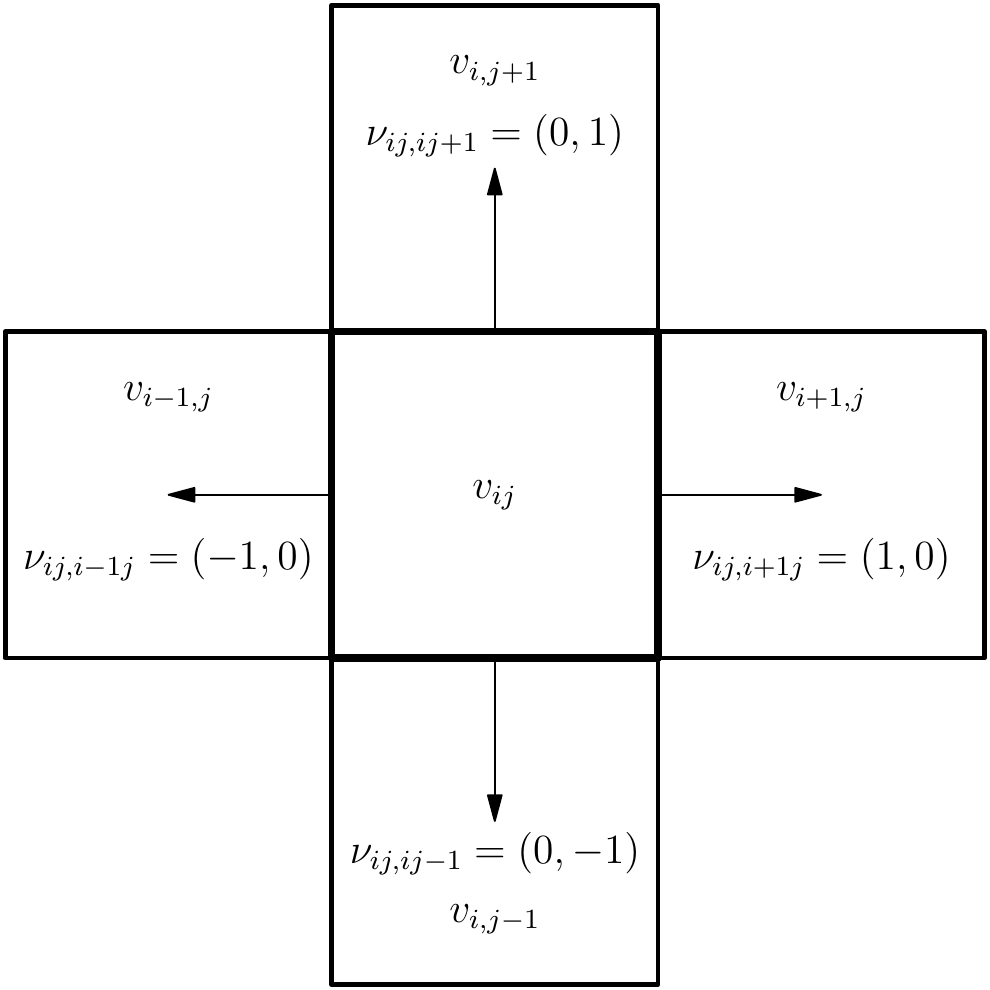}

a) \hspace{6cm} b)
}
\caption{(a) Dual mesh for the complementary finite volumes method - dots denote $\omega_h$ (\ref{def:numerical-grid-interior}), circles denote $\partial \omega_h$ (\ref{def:numerical-grid-boundary}) and solid lines stand for $V_h$ (\ref{dual-fvm-mesh}). (b) Notation for $\nu_{ij,\bar{i}\bar{j}}$.}
\label{fig:dual-mesh}
\end{figure}
We approximate the partial derivatives of $u$ on the boundary of the finite volume $v_{ij}$ as (here $\partial^h_{x_1}$ resp. $\partial^h_{x_2}$ stands for the finite difference w.r.t. $x_1$ resp. $x_2$):
\begin{eqnarray*}
\label{approx:partial-u-ijbaribarj-1}
\partial_{x_1}^h u^h_{ij,i+1j} &=& \frac{u^h_{i+1j} - u^h_{ij}}{h_1} \quad , \quad
\partial_{x_1}^h u^h_{ij,i-1j} = \frac{u^h_{ij} - u^h_{i-1j}}{h_1}, \\
\partial_{x_2}^h u^h_{ij,ij+1} &=& \frac{u^h_{ij+1} - u^h_{ij}}{h_2} \quad , \quad
\partial_{x_2}^h u^h_{ij,ij-1} = \frac{u^h_{ij} - u^h_{ij-1}}{h_2},
\end{eqnarray*}
and
\begin{eqnarray*}
\partial_{x_2}^h u^h_{ij,i+1j} &=& \frac{u^h_{ij,i+1j+1} - u^h_{ij,i+1j-1}}{h_2} \quad , \quad
\partial_{x_2}^h u^h_{ij,i-1j} = \frac{u^h_{ij,i-1j+1} - u^h_{ij,i-1j-1}}{h_2}, \\
\partial_{x_1}^h u^h_{ij,ij+1} &=& \frac{u^h_{ij,i+1j+1} - u^h_{ij,i-1j+1}}{h_1} \quad , \quad
\partial_{x_1}^h u_{ij,ij-1} = \frac{u_{ij,i+1j-1} - u_{ij,i-1j-1}}{h_1}, 
\end{eqnarray*}
where we denote
\begin{eqnarray*}
u^h_{ij,i+1j+1} &=& \frac{1}{4} \left( u^h_{ij} + u^h_{i+1j} + u^h_{ij+1} + u^h_{i+1j+1}\right) \\
u^h_{ij,i+1j-1} &=& \frac{1}{4} \left( u^h_{ij} + u^h_{i+1j} + u^h_{ij-1} + u^h_{i+1j-1}\right) \\
u^h_{ij,i-1j+1} &=& \frac{1}{4} \left( u^h_{ij} + u^h_{i-1j} + u^h_{ij+1} + u^h_{i-1j+1}\right) \\
u^h_{ij,i-1j-1} &=& \frac{1}{4} \left( u^h_{ij} + u^h_{i-1j} + u^h_{ij-1} + u^h_{i-1j-1}\right).
\end{eqnarray*}
The quantity $Q$ on the boundary of $\Gamma_{ij}$ is approximated as 
\begin{eqnarray*}
Q^h_{ij,i+1j} &=& \sqrt{1 + \left( \partial_{x_1}^h u^h_{ij,i+1j} \right)^2 + \left( \partial_{x_2}^h u^h_{ij,i+1j} \right)^2}, \\
Q^h_{ij,ij+1} &=& \sqrt{1 + \left( \partial_{x_1}^h u^h_{ij,ij+1} \right)^2 + \left( \partial_{x_2}^h u^h_{ij,ij+1} \right)^2}, \\
Q^h_{ij,i-1j} &=& \sqrt{1 + \left( \partial_{x_1}^h u^h_{ij,i-1j} \right)^2 + \left( \partial_{x_2}^h u^h_{ij,i-1j} \right)^2}, \\
Q^h_{ij,ij-1} &=& \sqrt{1 + \left( \partial_{x_1}^h u^h_{ij,ij-1} \right)^2 + \left( \partial_{x_2}^h u^h_{ij,ij-1} \right)^2}, 
\end{eqnarray*}
and $Q$ inside the finite volume $v_{ij}$ as
$$
Q^h_{ij} = \frac{1}{4} \left(Q^h_{ij,i+1j} + Q^h_{ij,ij+1} + Q^h_{ij,i-1j} + Q^h_{ij,ij-1}\right).
$$
Integrating (\ref{aniso-mean-curvature}) over a finite volume $\Omega_{ij}$ and applying the Stokes formula we obtain
\begin{equation}
\int_{\Omega_{ij}} H_\gamma \dx = \int_{\Omega_{ij}} \Div{\nabla_{\bfvec p} \gamma} \dx = \int_{\Gamma_{ij}} \nabla_{\bfvec p} \gamma \cdot \nu \dS.
\end{equation}
where $\nu$ denotes the outer unit normal vector to the finite volume boundary $\Gamma_{ij}$. We approximate the term on the left as
\begin{equation}
\label{aux:iso-mean-crv-1}
\int_{\Omega_{ij}} H_\gamma \dx \approx \mu\left(\Omega_{ij}\right) H^h_{\gamma,{ij}}
\end{equation}
and the term on the right as
\begin{equation}
\int_{\Gamma_{ij}} \nabla_{\bfvec p} \gamma \cdot \nu \dS = \sum_{v_{\bar{i}\bar{j}} \in \mathcal{N}_{ij}} \int_{\Gamma_{ij,\bar{i}\bar{j}}} \nabla_{\bfvec p} \gamma \cdot \nu \dS
\end{equation}
For the inner finite volume $v_{ij} \in V_h$, there are four different neighbors $v_{\bar{i}\bar{j}} \in \mathcal{N}_{ij}$ (see the Figure \ref{dual-fvm-mesh}). All the boundaries $\Gamma_{ij,\bar{i}\bar{j}}$ are linear segments and so $\nu = \nu_{ij,\bar{i}\bar{j}}$ is constant there. Moreover we assume that $\nabla_{\bfvec p} \gamma$ is constant along $\Gamma_{ij,\bar{i}\bar{j}}$ too. It gives
\begin{equation}
\label{aux:iso-mean-crv-2}
\sum_{v_{\bar{i}\bar{j}} \in \mathcal{N}_{ij}} \int_{\Gamma_{ij,\bar{i}\bar{j}}} \nabla_{\bfvec p} \gamma \cdot \nu \dS \approx
\sum_{v_{\bar{i}\bar{j}} \in \mathcal{N}_{ij}} l_{ij,\bar{i}\bar{j}} \nabla_{\bfvec p} \gamma_{ij,\bar{i}\bar{j}} \cdot \nu_{ij,\bar{i}\bar{j}} 
\end{equation}
where
\begin{equation}
\label{eqn:aniso-gamma-ij-eval}
\nabla_{\bfvec p} \gamma_{ij,\bar{i}\bar{j}} =
\left( \partial_{p_1} \gamma_{ij,\bar{i}\bar{j}}, \partial_{p_2} \gamma_{ij,\bar{i}\bar{j}} \right)^T 
= \left( \partial_{p_1} \gamma\left( \grad u^h_{ij,\bar{i}\bar{j}},-1 \right), \partial_{p_2} \gamma \left( \grad u^h_{ij,\bar{i}\bar{j}}, -1 \right) \right)^T.
\end{equation}
Putting (\ref{aux:iso-mean-crv-1}) and (\ref{aux:iso-mean-crv-2}) together we get
\begin{equation}
\label{aniso-mean-curvature-stokes-formula}
\mu\left({\Omega}\right) H_{\gamma,ij}^h = \sum_{v_{\bar{i}\bar{j}} \in \mathcal{N}_{ij}} l_{ij,\bar{i}\bar{j}} \nabla_{\bfvec p} \gamma_{ij,\bar{i}\bar{j}}  \nu_{ij,\bar{i}\bar{j}} \dS,
\end{equation}
For the regular dual mesh (\ref{dual-fvm-mesh}) we have
\begin{eqnarray}
H_{\gamma,ij}^h &=& \frac{1}{h_1 h_2}
\bigg(
h_2 \grad_{\bfvec p} \gamma_{ij,i+1j} \cdot \left(1,0\right)^T+
h_1 \grad_{\bfvec p} \gamma_{ij,ij+1} \cdot \left(0,1\right)^T  \nonumber \\
&+& h_2 \grad_{\bfvec p} \gamma_{ij,i-1j} \cdot \left(-1,0\right)^T+ 
h_1 \grad_{\bfvec p} \gamma_{ij,ij-1} \cdot \left(0,-1\right)^T
\bigg) \nonumber \\
&=& 
\label{eqn:co-fvm-aniso-mean-crv-graphs}
\left(\frac{\partial_{p_1}\gamma_{ij,i+1j} - \partial_{p_1}\gamma_{ij,i-1,j}}{h_1} + \frac{\partial_{p_2}\gamma_{ij,ij+1} - \partial_{p_2}\gamma_{ij,ij-1}}{h_2}\right),
\end{eqnarray}
and the approximation of $w_\gamma=Q H_\gamma$ on the finite volume $\Omega_{ij}$ is
\begin{equation}
w_{\gamma,ij}^h = Q^h_{ij} H_{\gamma,ij}^h.
\end{equation}
We set
\begin{eqnarray}
\label{approx:aniso-w-ij-bari-barj-1}
w^h_{\gamma,ij,i+1j} = \frac{1}{2} \left(w^h_{\gamma,ij} + w^h_{\gamma,i+1j}\right) \quad, \quad 
w^h_{\gamma,ij,ij+1} = \frac{1}{2} \left(w^h_{\gamma,ij} + w^h_{\gamma,ij+1}\right), \\
w^h_{\gamma,ij,i-1j} = \frac{1}{2} \left(w^h_{\gamma,ij} + w^h_{\gamma,i-1j}\right) \quad, \quad
\label{approx:aniso-w-ij-bari-barj-4}
w^h_{\gamma,ij,ij-1} = \frac{1}{2} \left(w^h_{\gamma,ij} + w^h_{\gamma,ij-1}\right).
\end{eqnarray}
Integrating (\ref{anisotropic-willmore-graphs-ut}) over $\Omega_{ij}$ and applying the Stokes theorem we get
\begin{eqnarray}
\label{aniso-willmore-u-t-stokes}
\int_{\Omega_{ij}} \frac{1}{Q} \Partial{u}{t} \dx &=&
-\int_{\Gamma_{ij}} \mathbb{E}_\gamma \grad w_\gamma \nu -
 \frac{1}{2} \frac{ w_\gamma^2}{Q^3} \Partial{u}{\nu} \dS 
\end{eqnarray}
where $\nu$ is the unit outer normal of the boundary $\Gamma_{ij}$. The left hand side is approximated as follows:
\begin{equation}
\label{aniso-willmore-u-t-stokes-lhs-approx}
\int_{\Omega_{ij}} \frac{1}{Q} \Partial{u}{t} \dx \approx 
 \frac{\mu\left(\Omega_{ij}\right)}{Q^h_{ij}} \frac{{\rm d}}{{\rm d}t} u^h_{ij} =
h_1 h_2 \frac{1}{Q^h_{ij}} \frac{{\rm d}}{{\rm d}t} u^h_{ij},
\end{equation}
where we assume that $u^h_{ij}$ and $Q^h_{ij}$ are constant on the element $v_{ij}$. For the integral on the right hand side of (\ref{aniso-willmore-u-t-stokes}) we have
\begin{eqnarray}
&&-\int_{\Gamma_{ij}} \mathbb{E}_\gamma \grad w_\gamma \nu -
 \frac{1}{2} \frac{ w^2_\gamma}{Q^3} \Partial{u}{\nu} \dS \approx \nonumber \\
&&- \sum_{v_{\bar{i}\bar{j}} \in \mathcal{N}_{ij}} l_{ij,\bar{i}\bar{j}} 
   \left( \mathbb{E}^h_{\gamma,ij,\bar{i}\bar{j}} \nabla w^h_{\gamma,ij,\bar{i}\bar{j}} \nu_{ij,\bar{i}\bar{j}} - 
          \frac{1}{2} \frac{\left(w^h_{\gamma,ij,\bar{i}\bar{j}}\right)^2}{\left( Q^h_{ij,\bar{i}\bar{j}}\right)^3} \nabla_{ij,\bar{i}\bar{j}} u^h_{ij} \nu_{ij,\bar{i}\bar{j}} \right).
\label{aniso-willmore-u-t-stokes-rhs-approx}
\end{eqnarray}
with the usual notation. Putting (\ref{aniso-willmore-u-t-stokes-lhs-approx}) and (\ref{aniso-willmore-u-t-stokes-rhs-approx}) together gives
$$
\frac{{\rm d}}{{\rm d}t} u^h_{ij} =
- \frac{Q^h_{ij}}{\mu\left(\Omega_{ij}\right)} \sum_{\bar{i}\bar{j} \in \mathcal{N}_{ij}} l_{ij,\bar{i}\bar{j}} 
   \left( \mathbb{E}^h_{\gamma,ij,\bar{i}\bar{j}} \nabla w^h_{\gamma,ij,\bar{i}\bar{j}} \nu_{ij,\bar{i}\bar{j}} - 
          \frac{1}{2} \frac{\left(w^h_{\gamma,ij,\bar{i}\bar{j}}\right)^2}{\left( Q^h_{ij,\bar{i}\bar{j}}\right)^3} \nabla_{ij\bar{i}\bar{j}} u^h_{ij} \nu_{ij,\bar{i}\bar{j}} \right).
$$
In the terms of the regular dual mesh (\ref{dual-fvm-mesh}) we get
\begin{eqnarray*}
\frac{{\rm d}}{{\rm d}t} u^h_{ij} &=&
- \frac{Q^h_{ij}}{h_1h_2}  
   \Bigg[h_2 \left( \mathbb{E}^h_{\gamma,ij,i+1j} \nabla w^h_{\gamma,ij,i+1j} \left(1,0\right)^T - 
          \frac{1}{2} \frac{\left(w^h_{\gamma,ij,i+1j}\right)^2}{\left( Q^h_{ij,i+1j}\right)^3} \nabla u^h_{ij,i+1j} \left(1,0\right)^T \right) \nonumber \\
&+& h_1 \left(  \mathbb{E}^h_{\gamma,ij,ij+1} \nabla w^h_{\gamma,ij,ij+1} \left(0,1\right)^T - 
          \frac{1}{2} \frac{\left(w^h_{\gamma,ij,ij+1}\right)^2}{\left( Q^h_{ij,ij+1}\right)^3} \nabla u^h_{ij,ij+1} \left(0,1\right)^T \right) \nonumber \\
&+& h_2 \left(  \mathbb{E}^h_{\gamma,ij,i-1j} \nabla w^h_{\gamma,ij,i-1j} \left(-1,0\right)^T - 
          \frac{1}{2} \frac{\left(w^h_{\gamma,ij,i-1j}\right)^2}{\left( Q^h_{ij,i-1j}\right)^3} \nabla u^h_{ij,i-1j} \left(-1,0\right)^T \right) \nonumber \\
\label{def:semidiscrete-aniso-willmore-ut}          
&+& h_1 \left(  \mathbb{E}^h_{\gamma,ij,ij-1} \nabla w^h_{\gamma,ij,ij-1} \left(0,-1\right)^T - 
          \frac{1}{2} \frac{\left(w^h_{\gamma,ij,ij-1}\right)^2}{\left( Q^h_{ij,ij-1}\right)^3} \nabla u^h_{ij,ij-1} \left(0,-1\right)^T \right) \Bigg].
\end{eqnarray*}
%
%
%
where for 
\begin{equation}
\label{approx:aniso-E-ij-bari-barj}
\mathbb{E}^h_{\gamma,ij,\bar{i}\bar{j}} = \left(
\begin{array}{cc}
\mathbbm{E}^h_{\gamma,11,ij,\bar{i}\bar{j}} & \mathbbm{E}^h_{\gamma,12,ij,\bar{i}\bar{j}} \\
\mathbbm{E}^h_{\gamma,21,ij,\bar{i}\bar{j}} & \mathbbm{E}^h_{\gamma,22,ij,\bar{i}\bar{j}}
\end{array}
\right),
\end{equation}
we have
\begin{eqnarray*}
\mathbb{E}^h_{\gamma,ij,i+1j} &=& \partial_{p_1} \partial_{p_1} \gamma\left( \grad u^h_{ij,i+1j}, -1 \right) \quad , \quad
\mathbb{E}^h_{\gamma,ij,ij+1} = \partial_{p_1} \partial_{p_2}  \gamma\left( \grad u^h_{ij,ij+1}, -1 \right), \\ 
\mathbb{E}^h_{\gamma,ij,i-1j} &=& \partial_{p_1} \partial_{p_1}  \gamma\left( \grad u^h_{ij,i-1j}, -1 \right) \quad , \quad
\mathbb{E}^h_{\gamma,ij,ij-1} =  \partial_{p_1} \partial_{p_2}  \gamma\left( \grad u^h_{ij,ij-1}, -1 \right). 
\end{eqnarray*}
The Neumann boundary condition $\partial_\nu u = 0$ on $\partial \Omega$ takes the following discrete form
\begin{eqnarray}
\label{fvm-aniso-willmore-neumann-bc-u-1}
{\rm if}\ i = 1\ {\rm then}\ \nu = \left(-1,0\right) &\Rightarrow& 
   \frac{1}{h_1}\left(u^h_{1,j}-u^h_{0,j}\right) = 0, \\
{\rm if}\ i = N_1-1\ {\rm then}\ \nu = \left(1,0\right) &\Rightarrow& 
   \frac{1}{h_1}\left(u^h_{N_1,j}-u^h_{N_1-1,j}\right) = 0, \\
{\rm if}\ j = 1\ {\rm then}\ \nu = \left(0,-1\right) &\Rightarrow& 
   \frac{1}{h_2} \left(u^h_{i,1}-u^h_{i,0}\right)=0,\\
\label{fvm-aniso-willmore-neumann-bc-u-2}
{\rm if}\ j = N_2-1\ {\rm then}\ \nu = \left(0,1\right) &\Rightarrow& 
   \frac{1}{h_2}\left(u^h_{i,N_2}-u^h_{i,N_2-1}\right) = 0
\end{eqnarray}
and from (\ref{aniso-willmore-graphs-neumann-bc}) we get
\begin{eqnarray}
\label{fvm-aniso-willmore-neumann-bc-w-1}
{\rm if}\ i = 1\ {\rm then}\ \nu = \left(-1,0\right) &\Rightarrow& 
   \mathbbm{E}_{\gamma,11,1j,0j} \partial_{x_1} w^h_{\gamma,1j,0j} +
   \mathbbm{E}_{\gamma,12,1j,0j} \partial_{x_2} w^h_{\gamma,1j,0j} = 0, \nonumber \\ \\
{\rm if}\ i = N_1-1\ {\rm then}\ \nu = \left(1,0\right) &\Rightarrow& 
     \mathbbm{E}_{\gamma,11,N_1-1j,N_1j} \partial_{x_1} w^h_{\gamma,N_1-1j,N_1j} + \nonumber \\
&&   \mathbbm{E}_{\gamma,12,N_1-1j,N_1j} \partial_{x_2} w^h_{\gamma,N_1-1j,N_1j} = 0, \\
{\rm if}\ j = 1\ {\rm then}\ \nu = \left(0,-1\right) &\Rightarrow& 
   \mathbbm{E}_{\gamma,21,i1,i0} \partial_{x_1} w^h_{\gamma,i1,i0} + 
   \mathbbm{E}_{\gamma,22,i1,i0} \partial_{x_2} w^h_{\gamma,i1,i0} = 0, \nonumber \\ \\ 
\label{fvm-aniso-willmore-neumann-bc-w-2}
{\rm if}\ j = N_2-1\ {\rm then}\ \nu = \left(0,1\right) &\Rightarrow& 
   \mathbbm{E}_{\gamma,21,iN_2-1,iN_2} \partial_{x_1} w^h_{\gamma,iN_2-1.iN_2} +  \nonumber \\
&&   \mathbbm{E}_{\gamma,22,iN_2-1,iN_2} \partial_{x_2} w^h_{\gamma,iN_2-1,N_2} = 0.
\end{eqnarray}

\subsection{The finite difference method}
We will prove the discrete version of the Theorem \ref{awg-energy-equality} in very similar manner as in \cite{Oberhuber-2006}. For this purpose we need to re-formulate the complementary finite volume scheme by the finite difference method. We complete the numerical grid with virtual nodes representing the finite volumes edges and we denote the new grid by $\varpi_h$. We index the nodes of $\varpi_h$ by indices $i \pm \frac 1{2}$ and $j \pm \frac 1{2}$. We define $u^h_{i \pm \frac 1{2} j \pm \frac 1{2}} = \frac 1{2} \left( u^h_{i \pm 1, j \pm 1} + u^h_{ij} \right)$ and denote by the uppercase a grid function with doubled indices i.e.
\begin{equation}
\label{double-ind}
\Uh_{kl} = u^h_{\frac k{2} \frac l{2}} \quad {\rm for} \quad k, l \in 0 \cdots 2N.
\end{equation}
We introduce the finite differences as follows:
\begin{eqnarray*}
\Uh_{f.,kl} = 2\frac{\Uh_{k+1,l} - \Uh_{kl}}{h} &\ , \ &
\Uh_{b.,kl} = 2\frac{\Uh_{kl}-\Uh_{k-1,l}}{h}, \\
\Uh_{.f,kl} = 2\frac{\Uh_{k,l+1} - \Uh_{kl}}{h}&,&
\Uh_{.b,kl} = 2\frac{\Uh_{kl}-\Uh_{k,l-1}}{h}, \\
\Uh_{c.,kl} = \frac{1}{2} \left(\Uh_{f.,kl}+\Uh_{b.,kl}\right)&,&
\Uh_{.c,kl} = \frac{1}{2} \left(\Uh_{.f,kl}+\Uh_{.b,kl}\right) 
\end{eqnarray*}
To approximate the gradient of $u$ we define
\begin{eqnarray*}
\gradh u^h_{ij} = \left( \Uh_{c.,2i,2j},\Uh_{.c,2i,2j} \right) \quad {\rm for} \quad i,j \in 1 \cdots N - 1
\end{eqnarray*}
Note that since we define $\gradh u^h_{ij}$ in terms of $\Uh_{kl}$ we can also write $\gradh u^h_{i \pm \frac 1{2},j \pm \frac 1{2}}$ for $i,j = 1 \cdots N - 1$.
%
%
The terms in (\ref{anisotropic-willmore-graphs-ut})-(\ref{anisotropic-willmore-graphs-w}) are approximated as follows
\begin{eqnarray*}
Q^h_{i \pm \frac 1{2}, j \pm \frac 1{2}} = \sqrt{1 + \left| \gradh u^h_{i \pm \frac 1{2}, j \pm \frac 1{2}} \right|^2}&,& 
Q^h_{ij} = \frac{1}{4} \sum_{\zeta,\eta \in \left\{ -1,1\right) \atop \left| \zeta \right| + \left| \eta \right| = 1} Q^h_{i+\frac \zeta{2},j+ \frac \eta{2}},\\
\mathbbm{E}^h_{\gamma,ij} = \mathbbm{E}_\gamma \left( u^h_{ij} \right)&,& \quad H^h_{\gamma,ij} = H_\gamma \left( u_{ij}^h \right)
\end{eqnarray*}
and the scheme has a form
\begin{eqnarray}
\label{discrete-willmore}
\frac{{\rm d} u^h}{{\rm d}t} &=& - Q^h_{ij} \gradh \cdot \left( \mathbbm{E}^h_{\gamma,ij} \gradh w^h_{\gamma,ij} - \frac{1}{2} \frac{\left( w^h_{\gamma,ij}\right)^2}{\left( Q^h_{ij}\right)^3} \gradh u^{h}_{ij} \right), \\
\label{discrete-willmore-w}
w^h_{\gamma,ij} &=& Q^h_{ij} H^h_{\gamma,ij}
\end{eqnarray}
Approximation of the Dirichlet boundary conditions is straightforward and the Neumann boundary conditions are handled as in (\ref{fvm-aniso-willmore-neumann-bc-u-1})--(\ref{fvm-aniso-willmore-neumann-bc-w-2}).
In what follows we define several discrete scalar products and we prove the discrete version of the Green formula. Assume having $f,g : \overline{\omega}_h \rightarrow \mathbbm{R}$, $\boldsymbol{f}: \overline{\omega}_h \rightarrow \mathbbm{R}^2$ and the related functions $F, G: \varpi_h \rightarrow \mathbbm{R}$, $\boldsymbol{F}: \varpi_h \rightarrow \mathbbm{R}^2$ given by (\ref{double-ind}) we define
\begin{eqnarray*}
\left[ F, G \right]_{pq}^{PQ} &=& \frac {h_1 h_2}{4} \sum_{k=p,l=q}^{P,Q} F_{kl} G_{kl}, \quad
\left( f, g \right)_h = \left( F,G \right)_h = \left[ F, G \right]_{11}^{2N-1,2N-1},\\
\left( F, G_{c.} \right)_c &=& \left[ F, G_{f.} \right]_{0,1}^{2N-1,2N-1} + \left[ F, G_{b.} \right]_{1,1}^{2N,2N-1}, \\
\left( F, G_{.c} \right)_c &=& \left[ F, G_{.f} \right]_{1,0}^{2N-1,2N-1} + \left[ F, G_{.b} \right]_{1,1}^{2N-1,2N}, \\
\left( \boldsymbol{F}, \gradh G \right)_c &=& \left( F^1, G_{c.} \right)_c + \left( F^2, G_{.c} \right)_c
\end{eqnarray*}
For the purpose of analysis, we will need the following grid version of the Green formula
\begin{lemma}
\label{lemma:discrete-green}
Let $u^h,v^h: \bar\omega_h \to {\mathbb R}$ and $U^h, V^h: \varpi_h \rightarrow \mathbbm{R}$ related to $u,v$ by (\ref{double-ind}). Then the Green formula is valid:
\begin{eqnarray*}
\label{discrete_green}
\left( \gradh u^h, v^h \right)_h &=& - \left( u^h, \gradh v^h \right)_c \\
&+& \frac {h_1}{2} \left( \sum_{l=1}^{2N-1} \left[  \left( \Uh_{2N-1,l}+\Uh_{2N,l} \right) \Vh_{2N,l} - \left( \Uh_{0l} + \Uh_{1l} \right) \Vh_{0l} \right]  \right) \\
&+& \frac{h_2}{2}\left( \sum_{k=1}^{2N-1}\left[  \left( \Uh_{k,2N-1}+\Uh_{k,2N} \right) \Vh_{k,2N} - \left( \Uh_{k0} + \Uh_{k1} \right) \Vh_{k0} \right] \right)
\end{eqnarray*}
\end{lemma}
\begin{proof}
It is easy to see that for fixed $k, l = 0, \cdots 2N$ the following relations holds.
\begin{eqnarray}
\label{per-partes-1}
\left[ \Uh_{b.}, V^h \right]_{1,l}^{2N-1,l} &=& - \left[ U, \Vh_{f.} \right]_{0,l}^{2N-1,l} + \frac{h_1}{2} \left( \Uh_{2N-1,l}\Vh_{2N,l} - \Uh_{0l}\Vh_{0l} \right),\\
\label{per-partes-2}
\left[ \Uh_{f.}, V^h \right]_{1,l}^{2N-1,l} &=& - \left[ U, \Vh_{b.} \right]_{1,l}^{2N,l}  + \frac{h_1}{2} \left( \Uh_{2N,l}\Vh_{2N,l} - \Uh_{1l}\Vh_{0l} \right), \\
\label{per-partes-3}
\left[ \Uh_{.b}, V^h \right]_{k,1}^{k,2N-1} &=& - \left[ U, \Vh_{.f} \right]_{k,0}^{k,2N-1} + \frac{h_2}{2} \left( \Uh_{k,2N-1}\Vh_{k,2N} - \Uh_{k0}\Vh_{k0} \right)\\
\label{per-partes-4}
\left[ \Uh_{.f}, V^h \right]_{k,1}^{k,2N-1} &=& - \left[ U, \Vh_{.b} \right]_{k,1}^{k,2N} + \frac{h_2}{2} \left( \Uh_{k,2N}\Vh_{k,2N} - \Uh_{k1}\Vh_{k0} \right)
\end{eqnarray}
Putting together (\ref{per-partes-1}) with (\ref{per-partes-2}) resp. (\ref{per-partes-3}) with (\ref{per-partes-4}) and summing over $l=1,\cdots 2N-1$ resp. over $k=1,\cdots 2N-1$ we obtain
\begin{eqnarray*}
\left( \Uh_{c.},V^h \right)_h &=& - \left( U, \Vh_{c.} \right)_c + \frac {h_1}{2} \sum_{l=1}^{2N-1} \left[  \left( \Uh_{2N-1,l}+\Uh_{2N,l} \right) \Vh_{2N,l} - \left( \Uh_{0l} + \Uh_{1l} \right) \Vh_{0l} \right],\\
\left( \Uh_{.c},V^h \right)_h &=& - \left( U, \Vh_{.c} \right)_c + \frac {h_2}{2} \sum_{k=1}^{2N-1} \left[  \left( \Uh_{k,2N-1}+\Uh_{k,2N} \right) \Vh_{k,2N} - \left( \Uh_{k0} + \Uh_{k1} \right) \Vh_{k0} \right],
\end{eqnarray*}
which gives
\begin{eqnarray*}
\left( \gradh U, V^h \right)_h &=& - \left( U, \gradh V \right)_c \\
&+&\frac {h_1}{2} \left( \sum_{l=1}^{2N-1} \left[  \left( \Uh_{2N-1,l}+\Uh_{2N,l} \right) \Vh_{2N,l} - \left( \Uh_{0l} + \Uh_{1l} \right) \Vh_{0l} \right]  \right) \\
&+& \frac{h_2}{2}\left( \sum_{k=1}^{2N-1}\left[  \left( \Uh_{k,2N-1}+\Uh_{k,2N} \right) \Vh_{k,2N} - \left( \Uh_{k0} + \Uh_{k1} \right) \Vh_{k0} \right] \right)
\end{eqnarray*}
\end{proof}

The following form of the previous Lemma \ref{lemma:discrete-green} will be more convenient for us:
\begin{lemma}
Let $p^h,u^h,v^h: \bar \omega_h \to {\mathbb R}$ and $v \mid_{\partial \omega} = 0$. Then
\begin{equation}
\label{zero_green}
\left( \gradh \cdot \left( p^h \gradh u^h \right), v^h \right)_h = - \left( p^h \gradh u^h, \gradh v^h \right)_c.
\end{equation}
\end{lemma}
\begin{theorem}
\label{energy_eq}
For solution of (\ref{discrete-willmore})--(\ref{discrete-willmore-w}) such that $u^h \mid_{\partial \omega_h} = 0$ and $w^h_\gamma=0 \mid_{\partial \omega_h}$ the following energy equality holds
$$\left( \left(u^h_t \right)^2, \frac{1}{Q^h} \right)_h + 
\frac{d}{dt}
     \left( \left(H^h_{\gamma}\right)^2, Q^h \right)_h = 0.$$
\end{theorem}
\begin{proof}
We start with the equation for $w^h_{\gamma,ij}$ (\ref{discrete-willmore-w}), divide by $Q^h_{ij}$, multiply by $\xi_{ij}$ vanishing on $\partial \omega_h$ and sum over $\omega$.
$$
\left( \frac{w^h_\gamma}{Q^h}, \xi \right)_h = \left( \gradh \cdot \left( \gamma_{p} \left( \gradh u^h, -1 \right) \right), \xi \right)_h.
$$
The Green theorem (\ref{zero_green}) gives
\begin{equation}
\label{w_xi_eq}
\left( \frac{w^h_\gamma}{Q^h}, \xi \right)_h = -\left( \gradh \xi, \nabla_p \left( \gradh u^h, -1 \right) \right).
\end{equation}
Taking the right hand side of (\ref{discrete-willmore}), multiplying by a test function $\varphi$ vanishing at $\partial \omega_h$ and applying the Green theorem (\ref{zero_green}) we obtain
\begin{eqnarray}
\label{u^h_t_phi_equation}
\left( - \gradh \cdot \left(
   2 \mathbb{E}^h \gradh w^h_\gamma - 
   \frac{\left(w^h_\gamma\right)^2}{\left(Q^h\right)^3} \gradh u^h \right), \varphi \right)_h = 
   \left(  
   2 \mathbb{E}^h \gradh w^h_\gamma - \frac{\left(w^h_\gamma\right)^2}{\left(Q^h\right)^3} \gradh u^h,
   \gradh \varphi \right)_c. \nonumber \\
\end{eqnarray}
Differentiating (\ref{w_xi_eq}) with respect to $t$ we obtain
\begin{eqnarray*}
&&\frac{d}{dt}\left( \frac{w^h_\gamma}{Q^h}, \xi \right)_h + 
\frac{d}{dt}\left( \nabla_{p_i} \gamma \left( \gradh u^h, -1 \right), \gradh \xi  \right)_c = \\
&&\frac{d}{dt} \left( \frac{w^h_\gamma}{Q^h}, \xi \right)_h + 
\left( \left( \nabla_p \otimes \nabla_p \right) \gamma \left( \gradh u^h, -1 \right) \gradh u^h_t, \gradh \xi  \right)_c \\
&=& \left( \frac{w^h_{\gamma,t}}{Q^h}, \xi \right)_h - \left( \frac{Q^h_t \cdot w^h_\gamma}{\left(Q^h\right)^2}, \xi \right)_h +
\left( \mathbb{E}^h \gradh u^h_t, \gradh \xi  \right)_c = 0.
\end{eqnarray*}
After substituting $\xi = w^h_\gamma$ we obtain
\begin{equation}
\label{w_equation}
\left( \frac{w^h_{\gamma,t}}{Q^h}, w^h_\gamma \right)_h - \left( \frac{Q^h_t}{\left(Q^h\right)^2}, \left(w^h_\gamma\right)^2 \right)_h + 
\left( \mathbb{E}^h \gradh u^h_t, \gradh w^h_\gamma \right)_c = 0,
\end{equation}
and a substitution $\varphi = u^h_t$ in (\ref{u^h_t_phi_equation}) gives
\begin{equation}
\label{u^h_t_equation}
\left( \left(u^h_t\right)^2, \frac{1}{Q^h} \right)_h  - 
\left( 2\mathbb{E}^h \gradh w^h_\gamma - \frac{\left(w^h_\gamma\right)^2}{\left(Q^h\right)^3} \gradh u^h, \gradh u^h_t \right)_c
 = 0.
\end{equation}
Substituting (\ref{w_equation}) to (\ref{u^h_t_equation}) (term $\mathbb{E}^h \gradh w^h$) we have
\begin{eqnarray}
\left( \left(u^h_t\right)^2, \frac{1}{Q^h} \right)_h + 
\left( \frac{w^h_{\gamma,t}}{Q^h}, w^h_\gamma \right)_h -
         \left( \frac{Q^h_t}{\left(Q^h \right)^2}, \left(w^h_\gamma\right)^2 \right)_h + \nonumber \\ 
   \left( \frac{\left(w^h_\gamma\right)^2}{\left(Q^h\right)^3}, \gradh u^h \cdot \gradh u^h_t \right)_c
   = 0.
\end{eqnarray}
We remind that $\gradh u^h \cdot \gradh u^h_t = Q^h \cdot Q^h_t$ which gives
\begin{eqnarray}
\left( \left(u^h_t\right)^2, \frac{1}{Q^h} \right)_h +
   \left( \frac{w^h_{\gamma,t}}{Q^h}, w^h_\gamma \right)_h -   
   \left( \frac{Q^h_t}{\left(Q^h\right)^2}, \left(w^h_\gamma\right)^2 \right)_h + 
   \left( \frac{\left(w^h_\gamma\right)^2}{\left(Q^h\right)^2}, Q^h_t \right)_c 
   = 0.\nonumber \\
\end{eqnarray}
It is equivalent to
\begin{equation}
\label{pre_stab_eq}
\left( \left(u^h_t\right)^2, \frac{1}{Q^h} \right)_h +
\left( \frac{w^h_{\gamma,t}}{\left(Q^h\right)^2}, w^h_\gamma \right)_h -
\left( \frac{Q^h_t}{\left(Q^h\right)^2}, \left(w^h_\gamma\right)^2 \right)_h + S_h = 0,
\end{equation}
for
\begin{eqnarray*}
S_h&=&
\frac{h_1^2}{4} \sum_{i=1}^{2N - 1} \left[ \left( \frac{w^h_{\gamma,i,2N}}{Q^h_{i,2N}} \right)^2 \cdot  Q^h_{t,i,2N} +
                           \left( \frac{w^h_{\gamma,i0}}{Q^h_{i0}} \right)^2 \cdot  Q^h_{t,i0} \right] + \\
&&\frac{h_2^2}{4} \sum_{j=1}^{2N - 1} \left[ \left( \frac{w^h_{\gamma,2Nj}}{Q^h_{2Nj}} \right)^2 \cdot Q^h_{t,2Nj} +
                           \left( \frac{w^h_{\gamma,0j}}{Q^h_{0j}} \right)^2 \cdot Q^h_{t,0j} \right].
\end{eqnarray*}
Finally from (\ref{pre_stab_eq}) we have
$$
\left( \left(u^h_t\right)^2, \frac{1}{Q^h} \right)_h +
\frac{d}{dt}\left( \left(H_{\gamma}^h\right)^2, Q^h \right)_h + S_h = 0.
$$
The term $S_h$ is vanishing on $\partial \omega$ because of the zero Dirichlet boundary conditions.
\end{proof}

\section{Time discretization}

The semi-discrete problem (\ref{def:semidiscrete-aniso-willmore-ut}) is discretize in time by the method of lines in the same way as in \cite{OberhuberSuzukiZabka-2011}. Here, the resulting system of ODEs is solved by the Runge-Kutta-Merson solver with adaptive integration time step implemented in CUDA to run on GPU.

\section{Experimental order of convergence}                         
To our best knowledge, there is not known any analytical solution for the Willmore flow of graphs even in the isotropic case. To be able to measure the experimental order of convergence we will solve a modified problem with additional forcing term $F$ having the following form:

\begin{eqnarray}
\label{anisotropic-willmore-graphs-ut-with-f}
\partial_t u &=&-Q\Div{ \mathbbm{E}_\gamma \grad w_\gamma - \frac{1}{2}\frac{w^2_\gamma}{Q^3} \grad u } + F\left(u\right)\ {\rm on}\ \left(0,T\right) \times \Omega,\\
\label{anisotropic-willmore-graphs-w-with-f}
w_\gamma &=& QH_\gamma \quad {\rm on}\ \left(0,T\right) \times \Omega,\\
\label{anisotropic-willmore-graphs-ini-with-f}
u \mid_{t=0} &=& u_0 \ {\rm on}\ \Omega, \\
\label{aniso-willmore-graphs-dirichlet-bc-with-f}
u &=& g,\ w_\gamma = 0  \quad {\rm on}\ \partial \Omega.
\end{eqnarray}
Having a function $\zeta\left(\bfvec x, t\right)$, we set
\begin{eqnarray*}
F_{W}\left(\zeta\right) &=&  Q\left(\zeta\right)\Div{ \frac{1}{Q\left(\zeta\right)} \mathbbm{P}\left(\zeta\right) \grad w_\gamma\left(\zeta\right) - \frac{1}{2}\frac{w^2_\gamma\left(\zeta\right)}{Q^3\left(\zeta\right)} \grad \zeta } + \Partial{\zeta}{t}\ {\rm on\ } \Omega \times \left[0,T\right], \\
w_\gamma\left(\zeta\right) &=& Q\left(\zeta\right) H_\gamma\left(\zeta\right)\ {\rm on\ } \Omega \times \left[0,T\right].
\end{eqnarray*}
As an analytical solution $\zeta(\bfvec x, t)$ of (\ref{anisotropic-willmore-graphs-ut-with-f})--(\ref{aniso-willmore-graphs-dirichlet-bc-with-f}) we chose the following function
\begin{equation}
\label{def:eoc-function-zeta}
\zeta\left(x,y,t\right) := \cos \left(\pi t\right) \frac{1}{r^{2n}} \left(x^n - r^n\right) \left(y^n - r^n\right) \exp\left( - \sigma \left(x^2 + y^2\right)\right)\ {\rm on}\ \Omega \times [0,T],
\end{equation}
for $\Omega \equiv [-r,r]^2$. For given $T$, we evaluate the errors in the norms of the spaces $L_1\left(\Omega;\left[0,T\right]\right)$, $L_2\left(\Omega;\left[0,T\right]\right)$ and $L_\infty\left(\Omega;\left[0,T\right]\right)$ resp. their approximations
\begin{eqnarray}
\label{eqn:l1-norm-eoc}
\left\| \uh - \mathcal{P}_h \left( \zeta \right) \right\|_{L_1\left(\omega_h;\left[0,T\right]\right)}^{h,\tau} &:=&
\sum_{i, j, k =0}^{N_1,N_2, M} \tau \left| \uh_{ij} \left(k\tau\right) - \zeta\left(-r+ih,-r+jh,k\tau\right) \right| h^2, \\
\left\| \uh  - \mathcal{P}_h \left( \zeta \right) \right\|_{L_2\left(\omega_h;\left[0,T\right]\right)}^{h,\tau} &:=&
\left( \sum_{i,j, k=0}^{N_1,N_2,M} \tau \left( \uh_{ij}\left(k\tau\right)  - \zeta\left(-r+ih,-r+jh,k\tau\right) \right)^2 h^2 \right)^{\frac{1}{2}},\nonumber \\ \\
\label{eqn:max-norm-eoc}
\left\| \uh  - \mathcal{P}_h \left( \zeta \right) \right\|_{L_\infty\left(\omega_h;\left[0,T\right]\right)}^{h,\tau} &:=&
\max_{\substack{i=0, \cdots,N_1 \\j=0, \cdots, N_2 \\k=0,\cdots, M }} \left| \uh_{ij}\left(k\tau\right)  - \zeta\left(-r+ih,-r+jh,k\tau\right) \right|,
\end{eqnarray}
for $\tau = T/M$. We would like to emphasize that $\tau$ does not correspond with the time step of the solver. The experimental order of convergence is evaluated as follows - for two approximations $u^{h_1}$ and $u^{h_2}$ obtained by the discretization with the space steps $h_1$ and $h_2$ we compute the approximation errors $Err_{h_1}$ and $Err_{h_2}$ in one norm of (\ref{eqn:l1-norm-eoc})--(\ref{eqn:max-norm-eoc}) and then we set
\begin{equation}
\label{eqn:eoc-def}
EOC\left(Err_{h_1},Err_{h_2}\right) := \frac{\log\left(Err_{h_1}/Err_{h_2}\right)}{\log\left(h_1/h_2\right)}.
\end{equation}
The results are presented in the Tables \ref{eoc-tab-1}--\ref{eoc-tab-8}. The solution of (\ref{anisotropic-willmore-graphs-ut-with-f})--(\ref{aniso-willmore-graphs-dirichlet-bc-with-f}) was approximated on the domain $\Omega \equiv \left[-4,4\right]^2$ and the time interval $\left[0,0.1\right]$. The Tables \ref{eoc-tab-1} -- \ref{eoc-tab-2} show EOC for $u$ and $w$ with the isotropic setting. Both unknown functions are approximated with the second order. The Tables \ref{eoc-tab-3} -- \ref{eoc-tab-4} show anisotropy (\ref{def:quadratic-form-anisotropy}) where 
\begin{equation}
\label{def:aniso-G-2-0-0-1}
\mathbbm{G}=\left( 
\begin{array}{cc}
2 & 0 \\
0 & 1
\end{array}
\right).
\end{equation}
Even in this case we obtained EOC equals to 2. As we can see in  the Tables \ref{eoc-tab-5} -- \ref{eoc-tab-6}, if we set
\begin{equation}
\label{def:aniso-G-2-1-1-1}
\mathbbm{G}=\left( 
\begin{array}{cc}
2 & 1 \\
1 & 1
\end{array}
\right),
\end{equation}
mixed partial derivatives are involved and EOC drops drops to 1 for both $u$ and $w_\gamma$. The last test is with the anisotropy given by (\ref{def:gamma2}) and we set $\epsilon_{abs}=0.1$. Results are presented in the Tables \ref{eoc-tab-7} and \ref{eoc-tab-8}. The non-linearity is much stronger here and it is not clear whether EOC is 1 or 2 from the Tables \ref{eoc-tab-7} and \ref{eoc-tab-8}. It would require computation on finer meshes which we were not able to do because of very long time needed for such simulation.
%
%
%
%
%
\begin{table}
\begin{center}
\begin{tabular}{|r|l|l|l|l|l|l|l|}\hline
\raisebox{-2ex}[0ex]{Meshes} &
\raisebox{-2ex}[0ex]{$h$} &
\multicolumn{2}{|c|}{ ${L_1\left( \Omega \right)}$}  &
\multicolumn{2}{|c|}{ ${L_2\left( \Omega \right)}$}  &
\multicolumn{2}{|c|}{ ${L_\infty\left( \Omega \right)}$}   \\
\cline{3-8}
         &           & Err.   & EOC        &  Err.  &  EOC        &  Err.  & EOC         \\  
\hline          
16       &  0.5	     & 0.09             &             &	0.2              &     	      &	1.05             &                \\
32       &  0.25     & \real{4.37}{-3}  & \raisebox{1ex}{\bf  4.5}  & \real{4.39}{-3}  &   \raisebox{1ex}{\bf 5.52} &	\real{2.82}{-2}  & \raisebox{1ex}{\bf 5.21}     \\
64       &  0.125    & \real{9.96}{-4}  & \raisebox{1ex}{\bf  2.13} &	\real{7.55}{-4}  & \raisebox{1ex}{\bf 2.54} &	\real{1.68}{-3}  & \raisebox{1ex}{\bf 4.06}     \\
128      &  0.0625   & \real{2.49}{-4}  & \raisebox{1ex}{\bf  1.99} &	\real{1.89}{-4}  & \raisebox{1ex}{\bf 1.99} &	\real{4.22}{-4}  & \raisebox{1ex}{\bf 1.99}     \\
256      &  0.03125  & \real{6.25}{-5}  & \raisebox{1ex}{\bf  1.99} &	\real{4.74}{-5}  & \raisebox{1ex}{\bf 1.00} &	\real{1.05}{-4}  & \raisebox{1ex}{\bf 1.99}     \\
\hline                                                                   
\end{tabular}                                                            
\end{center}
\caption{EOC of the approximation of the function $u$ with the anisotropy defined by (\ref{gamma-iso-gr}).}
\label{eoc-tab-1}
\end{table}
\begin{table}
\begin{center}
\begin{tabular}{|r|l|l|l|l|l|l|l|}\hline
\raisebox{-2ex}[0ex]{Meshes} &
\raisebox{-2ex}[0ex]{$h$} &
\multicolumn{2}{|c|}{ ${L_1\left( \Omega \right)}$}  &
\multicolumn{2}{|c|}{ ${L_2\left( \Omega \right)}$}  &
\multicolumn{2}{|c|}{ ${L_\infty\left( \Omega \right)}$}   \\
\cline{3-8}
         &           & Err.   & EOC        &  Err.    &  EOC        &  Err.    & EOC       \\  
\hline         
16       &  0.5	     & 0.48             &	     &  0.83            &	     & 3.65            &              \\
32       &  0.25     & \real{3.19}{-2}  & \raisebox{1ex}{\bf 3.92} & \real{4.89}{-2}  & \raisebox{1ex}{\bf 4.09} & 0.53            & \raisebox{1ex}{\bf 2.77}   \\
64       &  0.125    & \real{7.02}{-3}  & \raisebox{1ex}{\bf 2.18} & \real{7.79}{-3}  & \raisebox{1ex}{\bf 2.65} & \real{8.81}{-2} & \raisebox{1ex}{\bf 2.59}   \\
128      &  0.0625   & \real{1.76}{-3}  & \raisebox{1ex}{\bf 1.99} & \real{1.91}{-3}  & \raisebox{1ex}{\bf 2.02} & \real{2.1}{-2}  & \raisebox{1ex}{\bf 2.06}   \\
256      &  0.03125  & \real{4.40}{-4}  & \raisebox{1ex}{\bf 2}    & \real{4.75}{-4}  & \raisebox{1ex}{\bf 2}    & \real{5.2}{-3}  & \raisebox{1ex}{\bf 2.01}   \\
\hline                                                                   
\end{tabular}                                                            
\end{center}
\caption{EOC of the approximation of the function $w_\gamma$ with the anisotropy defined by (\ref{gamma-iso-gr}).}
\label{eoc-tab-2}
\end{table}
%
%
%
%
\begin{table}
\begin{center}
\begin{tabular}{|r|l|l|l|l|l|l|l|}\hline
\raisebox{-2ex}[0ex]{Meshes} &
\raisebox{-2ex}[0ex]{$h$} &
\multicolumn{2}{|c|}{ ${L_1\left( \Omega \right)}$}  &
\multicolumn{2}{|c|}{ ${L_2\left( \Omega \right)}$}  &
\multicolumn{2}{|c|}{ ${L_\infty\left( \Omega \right)}$}   \\
\cline{3-8}
         &           & Err.    & EOC                      &  Err.   &  EOC                      &  Err.  & EOC         \\  
\hline          
16       &  0.5	     &  0.7    &                          &  1.2    &	                        &  5.5   &             \\
32       &  0.25     &  0.53   & \raisebox{1ex}{\bf 0.4}  &  0.93   & \raisebox{1ex}{\bf  0.37} &  4.9   & \raisebox{1ex}{\bf 0.17}  \\
64       &  0.125    &  0.085  & \raisebox{1ex}{\bf 2.6}  &  0.24   & \raisebox{1ex}{\bf  2}    &  1.5   & \raisebox{1ex}{\bf 1.8}     \\
128      &  0.0625   &  0.0059 & \raisebox{1ex}{\bf 3.8}  &  0.013  & \raisebox{1ex}{\bf  4.2}  &  0.089 & \raisebox{1ex}{\bf 4}   \\
256      &  0.03125  &  0.0015 & \raisebox{1ex}{\bf 2}    &  0.0035 & \raisebox{1ex}{\bf  1.9}  &  0.024 & \raisebox{1ex}{\bf 1.9}   \\
\hline                                                                   
\end{tabular}                                                            
\end{center}
\caption{EOC of the approximation of the function $u$ with the anisotropy defined by (\ref{def:quadratic-form-anisotropy}) and (\ref{def:aniso-G-2-0-0-1}).}
\label{eoc-tab-3}
\end{table}
\begin{table}
\begin{center}
\begin{tabular}{|r|l|l|l|l|l|l|l|}\hline
\raisebox{-2ex}[0ex]{Meshes} &
\raisebox{-2ex}[0ex]{$h$} &
\multicolumn{2}{|c|}{ ${L_1\left( \Omega \right)}$}  &
\multicolumn{2}{|c|}{ ${L_2\left( \Omega \right)}$}  &
\multicolumn{2}{|c|}{ ${L_\infty\left( \Omega \right)}$}   \\
\cline{3-8}
         &           & Err.   & EOC                      &  Err.    &  EOC                      &  Err.        & EOC                           \\  
\hline         
16       &  0.5	     & 3.5    &	                         &    7.3   &     	                &	41     &                               \\
32       &  0.25     & 3.2    &	\raisebox{1ex}{\bf 0.15} &    7.3   & \raisebox{1ex}{\bf  0}    &	44     & \raisebox{1ex}{\bf -0.1}      \\
64       &  0.125    & 0.58   &	\raisebox{1ex}{\bf  2.4} &    1.6   & \raisebox{1ex}{\bf  2.1}  &	19     & \raisebox{1ex}{\bf  1.2}      \\
128      &  0.0625   & 0.04   &	\raisebox{1ex}{\bf  3.9} &    0.085 & \raisebox{1ex}{\bf  4.3}  &	0.5    & \raisebox{1ex}{\bf  5.2}      \\
256      &  0.03125  & 0.01   &	\raisebox{1ex}{\bf  1.9} &    0.023 & \raisebox{1ex}{\bf  1.9}  &	0.14   & \raisebox{1ex}{\bf  1.9}      \\
\hline                                                                   
\end{tabular}                                                            
\end{center}
\caption{EOC of the approximation of the function $w_\gamma$ with the anisotropy defined by (\ref{def:quadratic-form-anisotropy}) and (\ref{def:aniso-G-2-0-0-1}).}
\label{eoc-tab-4}
\end{table}
%
%
%
%
\begin{table}
\begin{center}
\begin{tabular}{|r|l|l|l|l|l|l|l|}\hline
\raisebox{-2ex}[0ex]{Meshes} &
\raisebox{-2ex}[0ex]{$h$} &
\multicolumn{2}{|c|}{ ${L_1\left( \Omega \right)}$}  &
\multicolumn{2}{|c|}{ ${L_2\left( \Omega \right)}$}  &
\multicolumn{2}{|c|}{ ${L_\infty\left( \Omega \right)}$}   \\
\cline{3-8}
         &           & Err.   & EOC                      &  Err.    &  EOC                      &  Err.   & EOC                             \\  
\hline          
16       &  0.5	     & 2.1    &	                         &   4      &     	                &   19     &                                \\
32       &  0.25     & 0.46   &	\raisebox{1ex}{\bf  2.2} &   0.64   & \raisebox{1ex}{\bf 2.6}   &    2.9   & \raisebox{1ex}{\bf  2.7}       \\
64       &  0.125    & 0.15   &	\raisebox{1ex}{\bf  1.6} &   0.34   & \raisebox{1ex}{\bf 0.92}  &    1.9   & \raisebox{1ex}{\bf  0.59}      \\
128      &  0.0625   & 0.06   &	\raisebox{1ex}{\bf  1.3} &   0.16   & \raisebox{1ex}{\bf 1.1}   &    1.1   & \raisebox{1ex}{\bf  0.74}      \\
256      &  0.03125  & 0.025  &	\raisebox{1ex}{\bf  1.2} &   0.08   & \raisebox{1ex}{\bf 0.99}  &    0.72  & \raisebox{1ex}{\bf  0.66}      \\
\hline                                                                   
\end{tabular}                                                            
\end{center}
\caption{EOC of the approximation of the function $u$ with the anisotropy defined by (\ref{def:quadratic-form-anisotropy}) and (\ref{def:aniso-G-2-1-1-1}).}
\label{eoc-tab-5}
\end{table}
\begin{table}
\begin{center}
\begin{tabular}{|r|l|l|l|l|l|l|l|}\hline
\raisebox{-2ex}[0ex]{Meshes} &
\raisebox{-2ex}[0ex]{$h$} &
\multicolumn{2}{|c|}{ ${L_1\left( \Omega \right)}$}  &
\multicolumn{2}{|c|}{ ${L_2\left( \Omega \right)}$}  &
\multicolumn{2}{|c|}{ ${L_\infty\left( \Omega \right)}$}   \\
\cline{3-8}
         &           & Err.   & EOC                      &  Err.    &  EOC                      &  Err.    & EOC                        \\  
\hline         
16       &  0.5	     & 12    &	                         &  27      &	                         & 120   &                              \\
32       &  0.25     & 3.3   &	\raisebox{1ex}{\bf 1.9}  &   5      & \raisebox{1ex}{\bf  2.4}   &  30   & \raisebox{1ex}{\bf 2}        \\
64       &  0.125    & 1.2   &	\raisebox{1ex}{\bf 1.5}  &   2.7    & \raisebox{1ex}{\bf  0.9}   &  25   & \raisebox{1ex}{\bf 0.26}     \\
128      &  0.0625   & 0.53  &	\raisebox{1ex}{\bf 1.1}  &   1.3    & \raisebox{1ex}{\bf  0.99}  & 13    & \raisebox{1ex}{\bf 0.92}     \\
256      &  0.03125  & 0.25  &	\raisebox{1ex}{\bf 1.1}  &   0.76   & \raisebox{1ex}{\bf  0.82}  &  8.3  & \raisebox{1ex}{\bf 0.68}     \\
\hline                                                                   
\end{tabular}                                                            
\end{center}
\caption{EOC of the approximation of the function $w_\gamma$ with the anisotropy defined by (\ref{def:quadratic-form-anisotropy}) and (\ref{def:aniso-G-2-1-1-1}).}
\label{eoc-tab-6}
\end{table}
%
%
%
%
\begin{table}
\begin{center}
\begin{tabular}{|r|l|l|l|l|l|l|l|}\hline
\raisebox{-2ex}[0ex]{Meshes} &
\raisebox{-2ex}[0ex]{$h$} &
\multicolumn{2}{|c|}{ ${L_1\left( \Omega \right)}$}  &
\multicolumn{2}{|c|}{ ${L_2\left( \Omega \right)}$}  &
\multicolumn{2}{|c|}{ ${L_\infty\left( \Omega \right)}$}   \\
\cline{3-8}
         &           & Err.   & EOC        &  Err.  &  EOC        &  Err.  & EOC         \\  
\hline          
16       &  0.5	     &  2.9   &	                         &  3.2   &	                      & 9     &             \\
32       &  0.25     &  1.8   &	\raisebox{1ex}{\bf 0.72} &  1.7   & \raisebox{1ex}{\bf 0.9}   & 4.4   & \raisebox{1ex}{\bf 1}   \\
64       &  0.125    &  0.46  &	\raisebox{1ex}{\bf 1.9}  &  1.5   & \raisebox{1ex}{\bf 0.17}  & 15    & \raisebox{1ex}{\bf -1.8}  \\
128      &  0.0625   &  0.078 &	\raisebox{1ex}{\bf 2.6}  &  0.11  & \raisebox{1ex}{\bf 3.7}   & 0.55  & \raisebox{1ex}{\bf 4.8}   \\
256      &  0.03125  &  0.027 &	\raisebox{1ex}{\bf 1.47} &  0.045 & \raisebox{1ex}{\bf 1.28 } & 0.18  & \raisebox{1ex}{\bf 1.6}   \\
\hline                                                                   
\end{tabular}                                                            
\end{center}
\caption{EOC of the approximation of the function $u$ with the anisotropy defined by (\ref{def:gamma2}) and $\epsilon_{abs}=1$.}
\label{eoc-tab-7}
\end{table}
\begin{table}
\begin{center}
\begin{tabular}{|r|l|l|l|l|l|l|l|}\hline
\raisebox{-2ex}[0ex]{Meshes} &
\raisebox{-2ex}[0ex]{$h$} &
\multicolumn{2}{|c|}{ ${L_1\left( \Omega \right)}$}  &
\multicolumn{2}{|c|}{ ${L_2\left( \Omega \right)}$}  &
\multicolumn{2}{|c|}{ ${L_\infty\left( \Omega \right)}$}   \\
\cline{3-8}
         &           & Err.   & EOC                        &  Err.    &  EOC                      &  Err.       & EOC       \\  
\hline         
16       &  0.5	     & 14     &	                           &   14     &                           &	41     &             \\
32       &  0.25     &  8.6   &	\raisebox{1ex}{\bf 0.69}   &   8.3    & \raisebox{1ex}{\bf 0.74}  &	29     & \raisebox{1ex}{\bf 0.49}      \\
64       &  0.125    &  9     &	\raisebox{1ex}{\bf -0.067} &   36     & \raisebox{1ex}{\bf -2.1}  &	460    & \raisebox{1ex}{\bf -4}      \\
128      &  0.0625   &  0.69  &	\raisebox{1ex}{\bf 3.7}    &  1.1     & \raisebox{1ex}{\bf 5}     &	5.3    & \raisebox{1ex}{\bf 6.4}      \\
256      &  0.03125  &  0.29  &	\raisebox{1ex}{\bf 1.25}   &  0.49    & \raisebox{1ex}{\bf 1.16}  &	1.49   & \raisebox{1ex}{\bf 1.8}      \\
\hline                                                                   
\end{tabular}                                                            
\end{center}
\caption{EOC of the approximation of the function $w_\gamma$ with the anisotropy defined by (\ref{def:gamma2}) and $\epsilon_{abs}=1$.}
\label{eoc-tab-8}
\end{table}

\section{Numerical experiments}
On the Figures \ref{figure-1}--\ref{figure-3} we show results of qualitative analysis. The initial condition is $u_0 =  \sin\left( 3 \pi \sqrt{ x^2 + y^2} \right) \quad \rm{on}\ \Omega \equiv \left[-2,2\right]^2$ and we set  the Neumann boundary conditions $\Partial{\varphi}{\nu} = \mathbbm{E}_\gamma \grad w_\gamma \cdot \nu = 0  \quad \rm{on}\ \partial \Omega$. The computational domain is covered by $100 \times 100$ meshes. The Figure \ref{figure-1} depicts result obtained with the anisotropy given by (\ref{def:quadratic-form-anisotropy}) and
\begin{equation}
\label{def:aniso-G-8-0-0-1}
\mathbbm{G} := \left( 
\begin{array}{cc}
8 & 0 \\
0 & 1
\end{array}
\right)
\end{equation}
on the time interval $\left[0,0.001\right]$ .
\begin{figure}
\center{
\includegraphics[width=6cm]{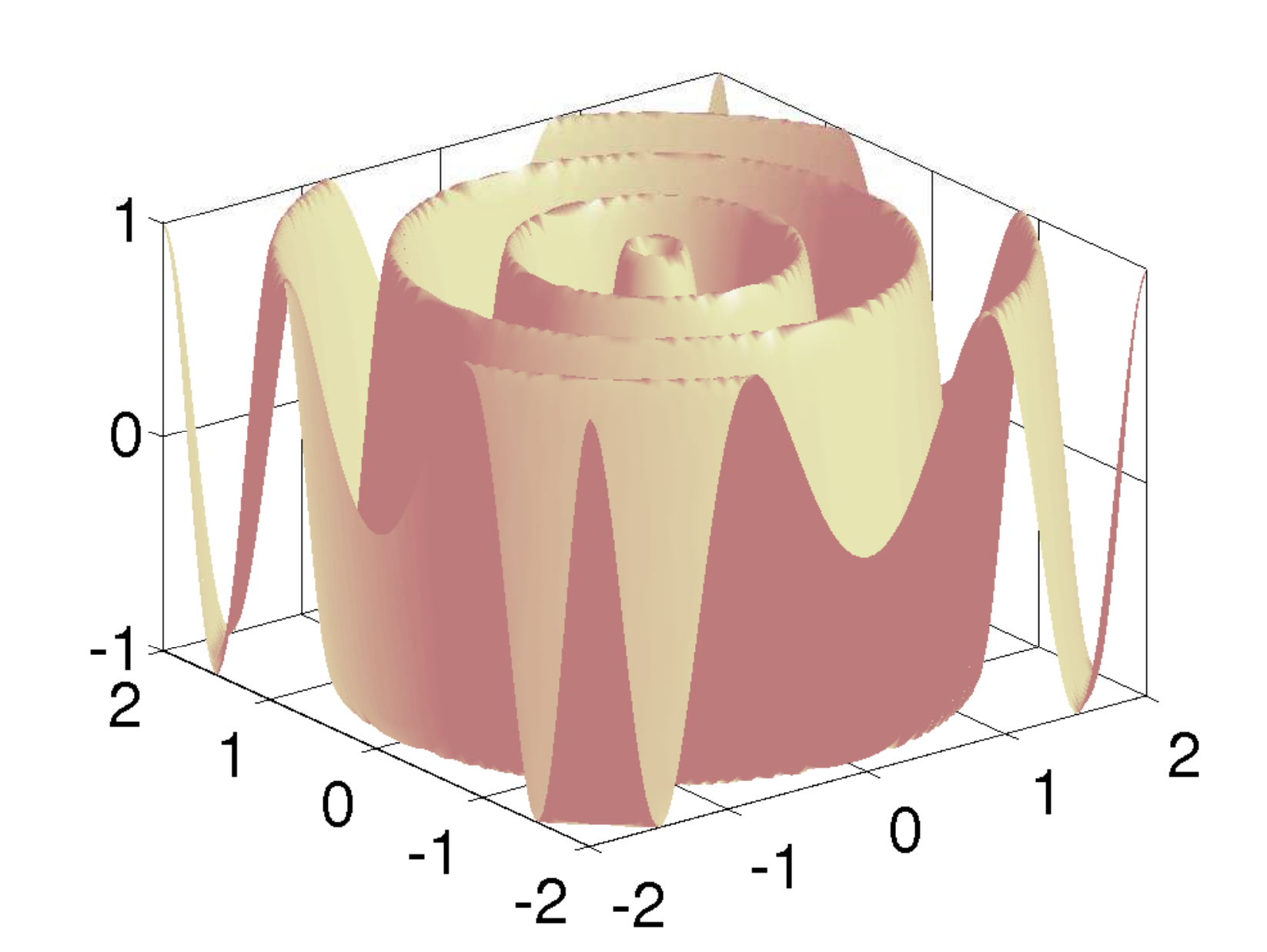}
\includegraphics[width=6cm]{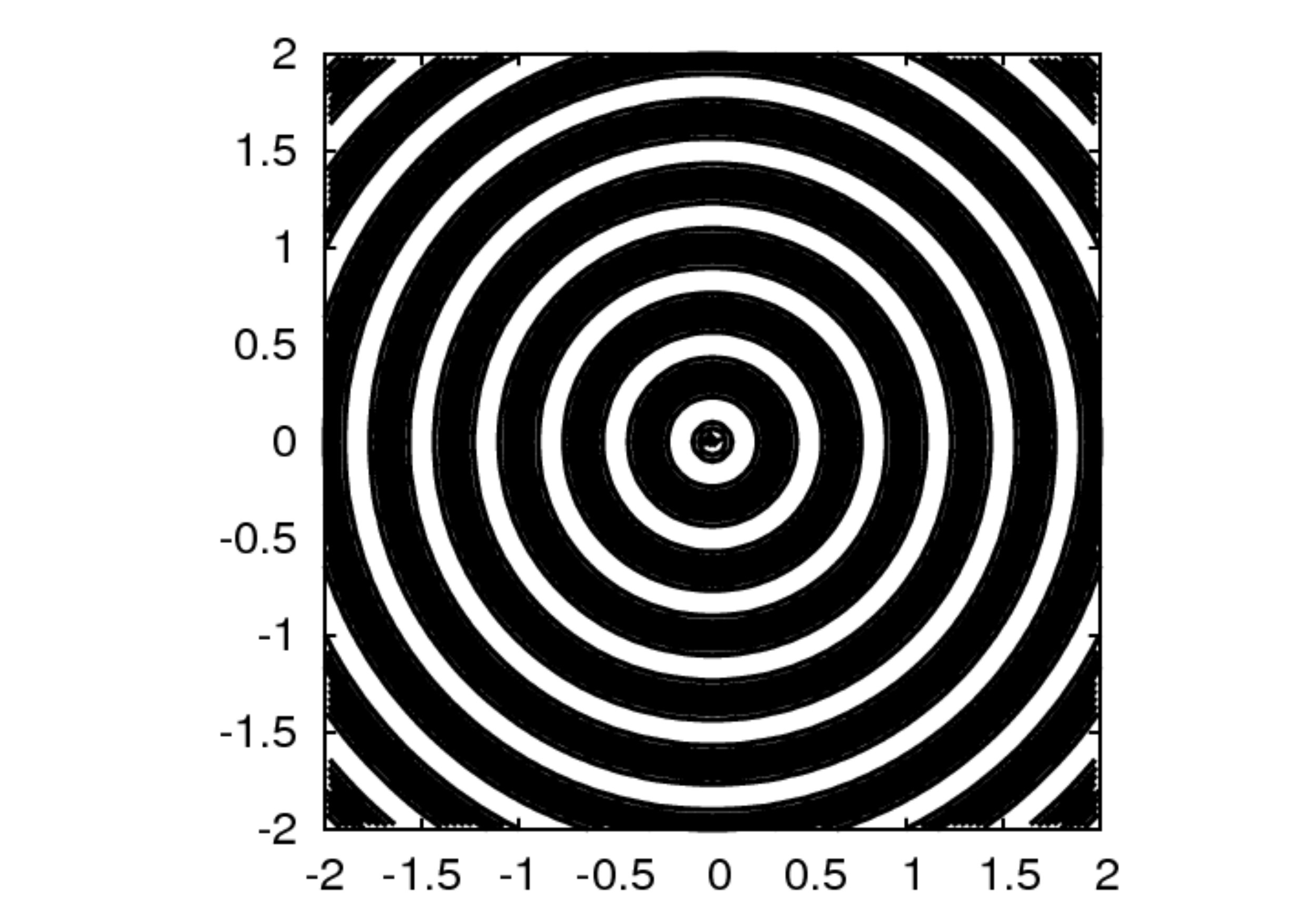}
\includegraphics[width=6cm]{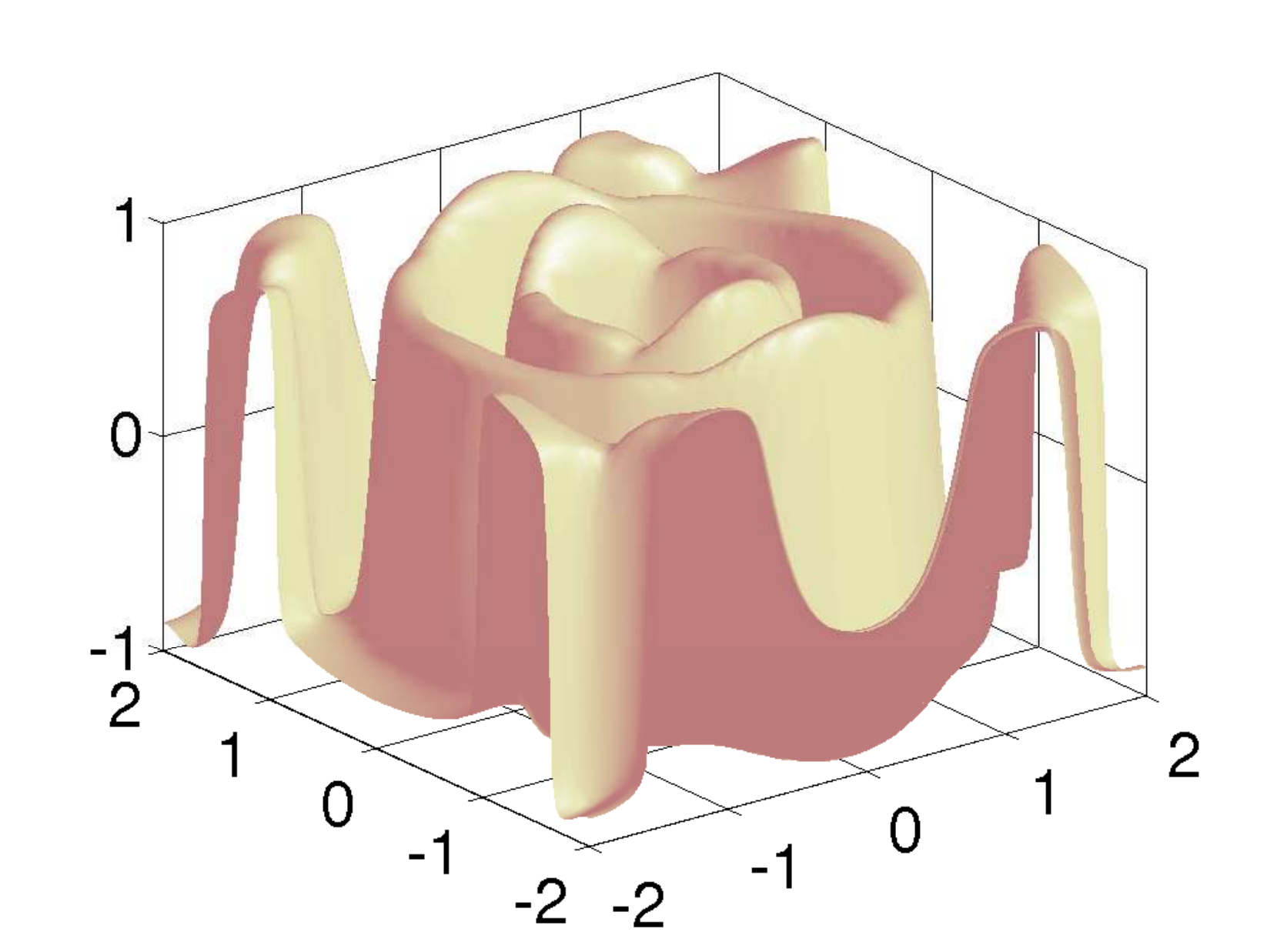}
\includegraphics[width=6cm]{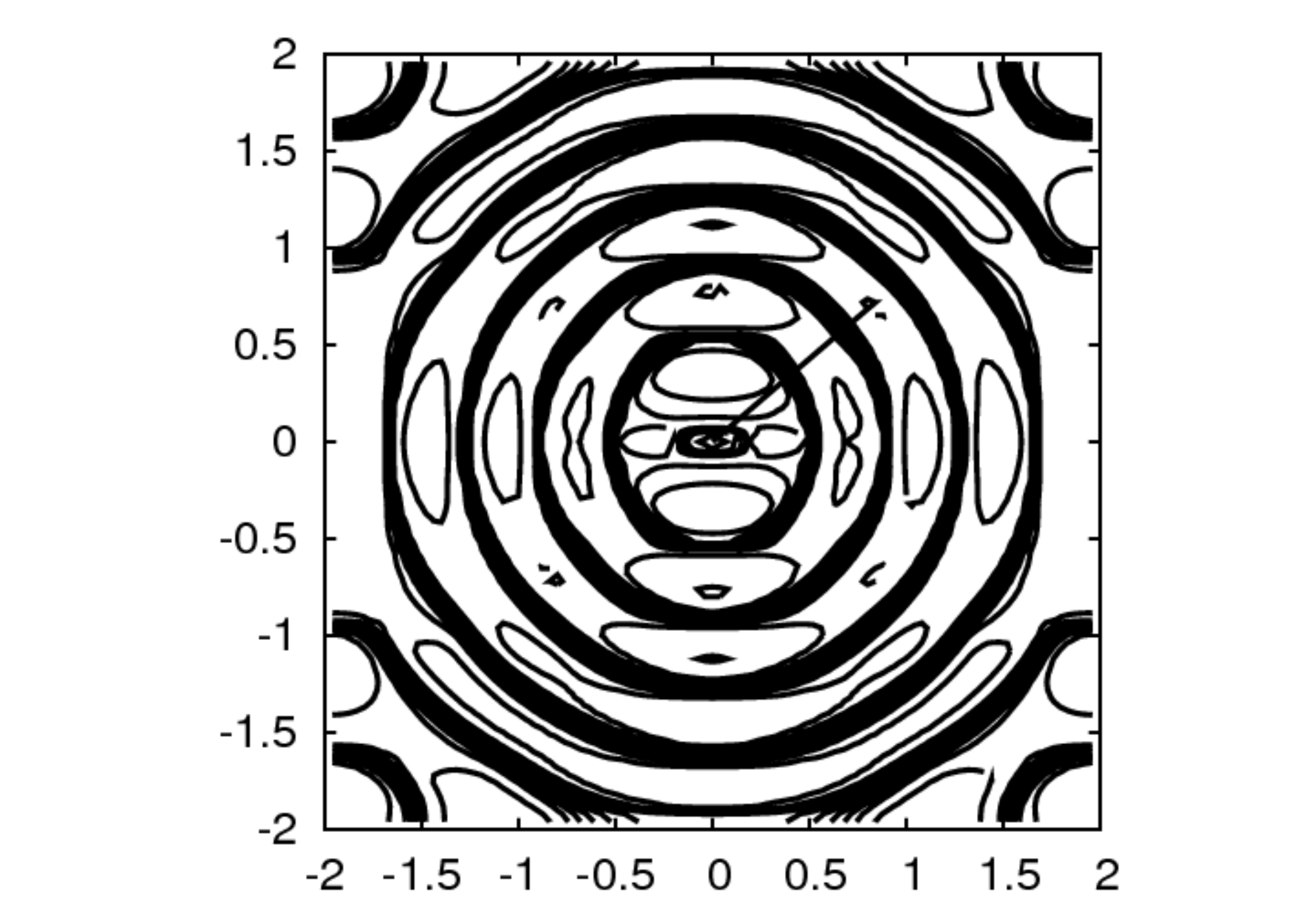}
\includegraphics[width=6cm]{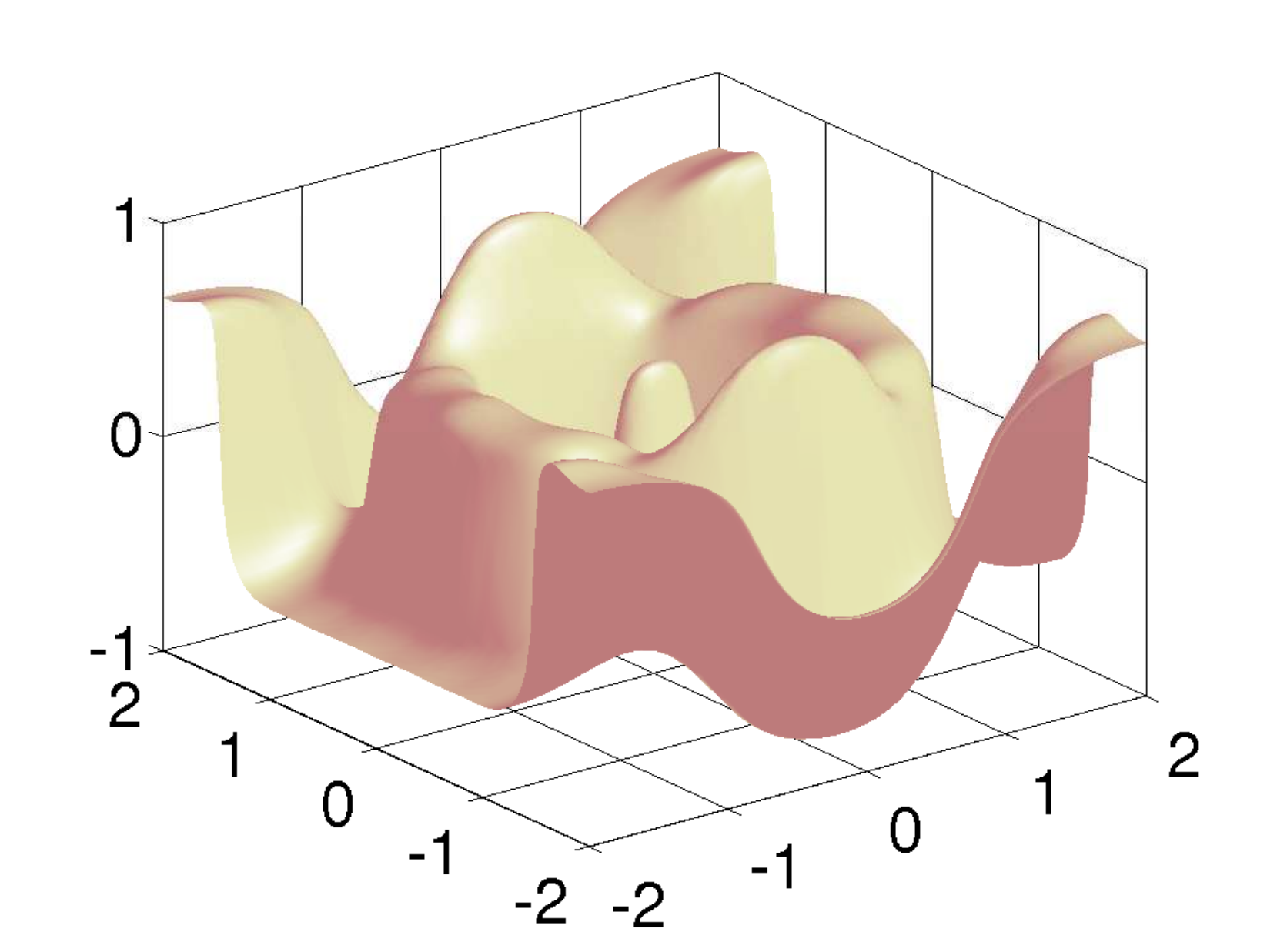}
\includegraphics[width=6cm]{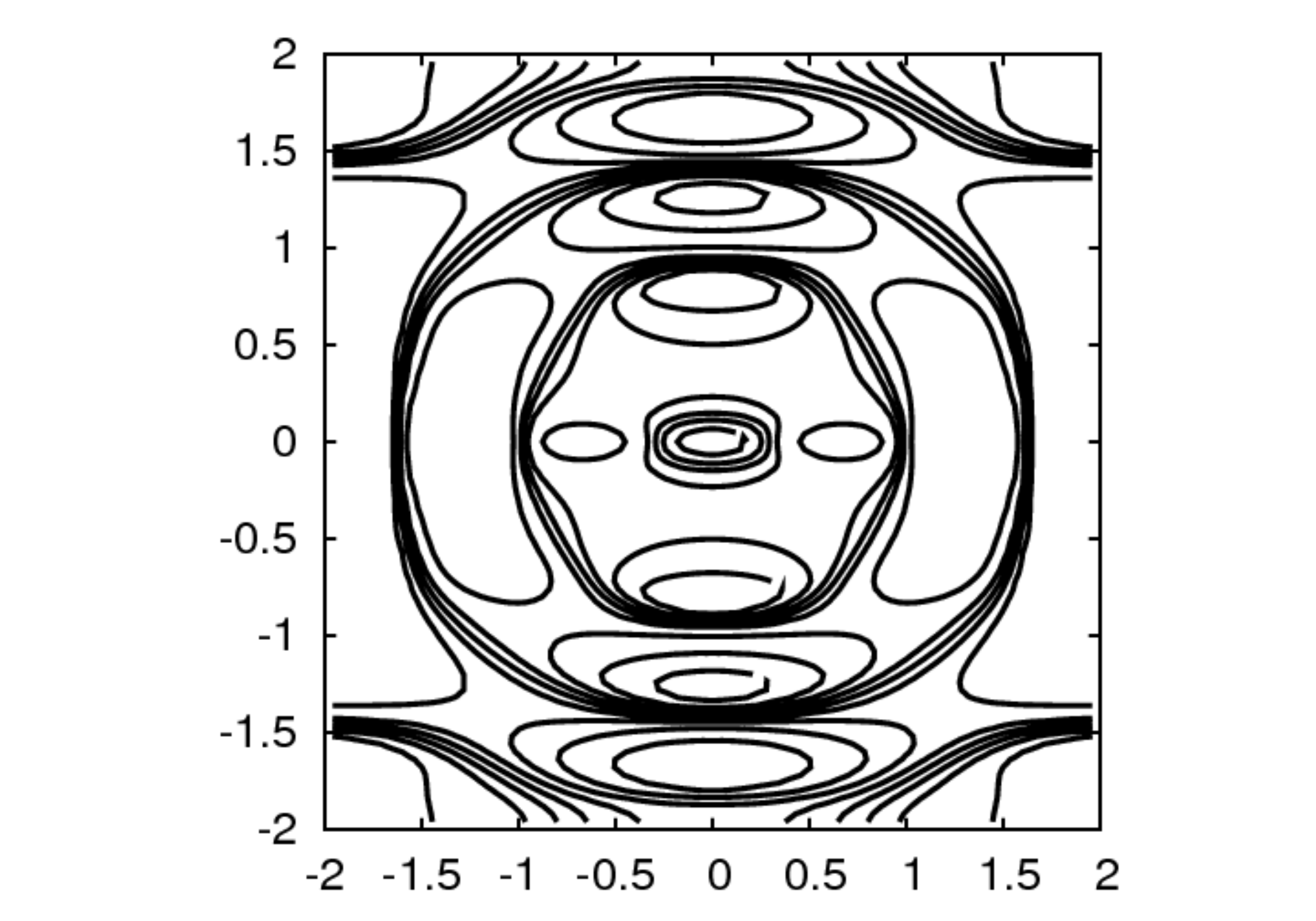}
\includegraphics[width=6cm]{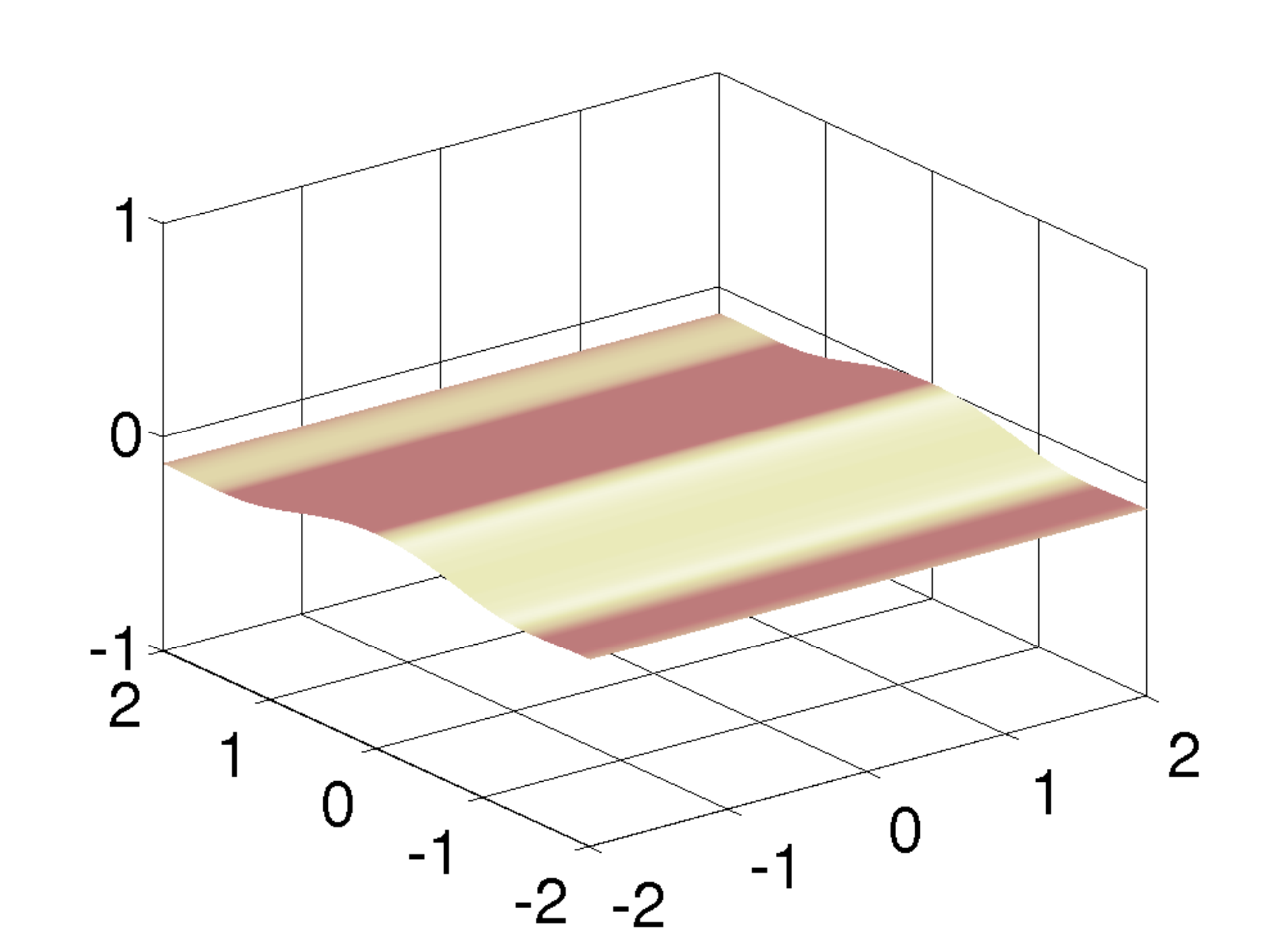}
\includegraphics[width=6cm]{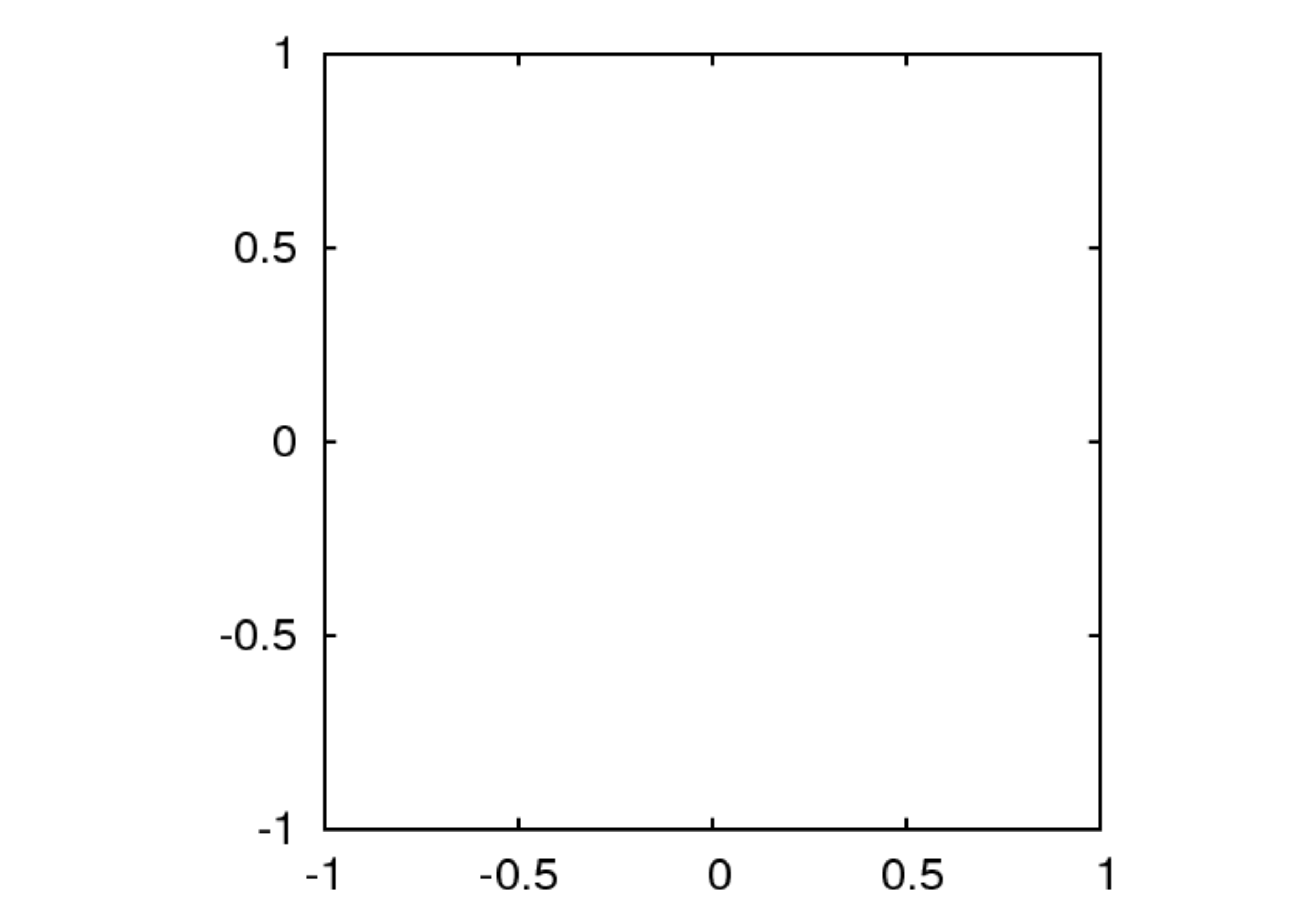}
}
\caption{Results obtained with the anisotropy given by (\ref{def:quadratic-form-anisotropy}) and (\ref{def:aniso-G-8-0-0-1}) at times $t=0$, $t=1.6 \cdot 10^{-5}$, $t=1.28 \cdot 10^{-4}$ and $t=0.001$ (graph of $u$ on the left, level-lines of $u$ on the right)}
\label{figure-1}
\end{figure}
The Figure \ref{figure-2} show the same anisotropy (\ref{def:quadratic-form-anisotropy}) but with 
\begin{equation}
\mathbbm{G} := \left( 
\label{def:aniso-G-10-8-8-10}
\begin{array}{cc}
10 & 8 \\
8 & 10
\end{array}
\right)
\end{equation}
on the time interval $\left[0,\real{1.024}{-3}\right]$.
\begin{figure}
\center{
\includegraphics[width=6cm]{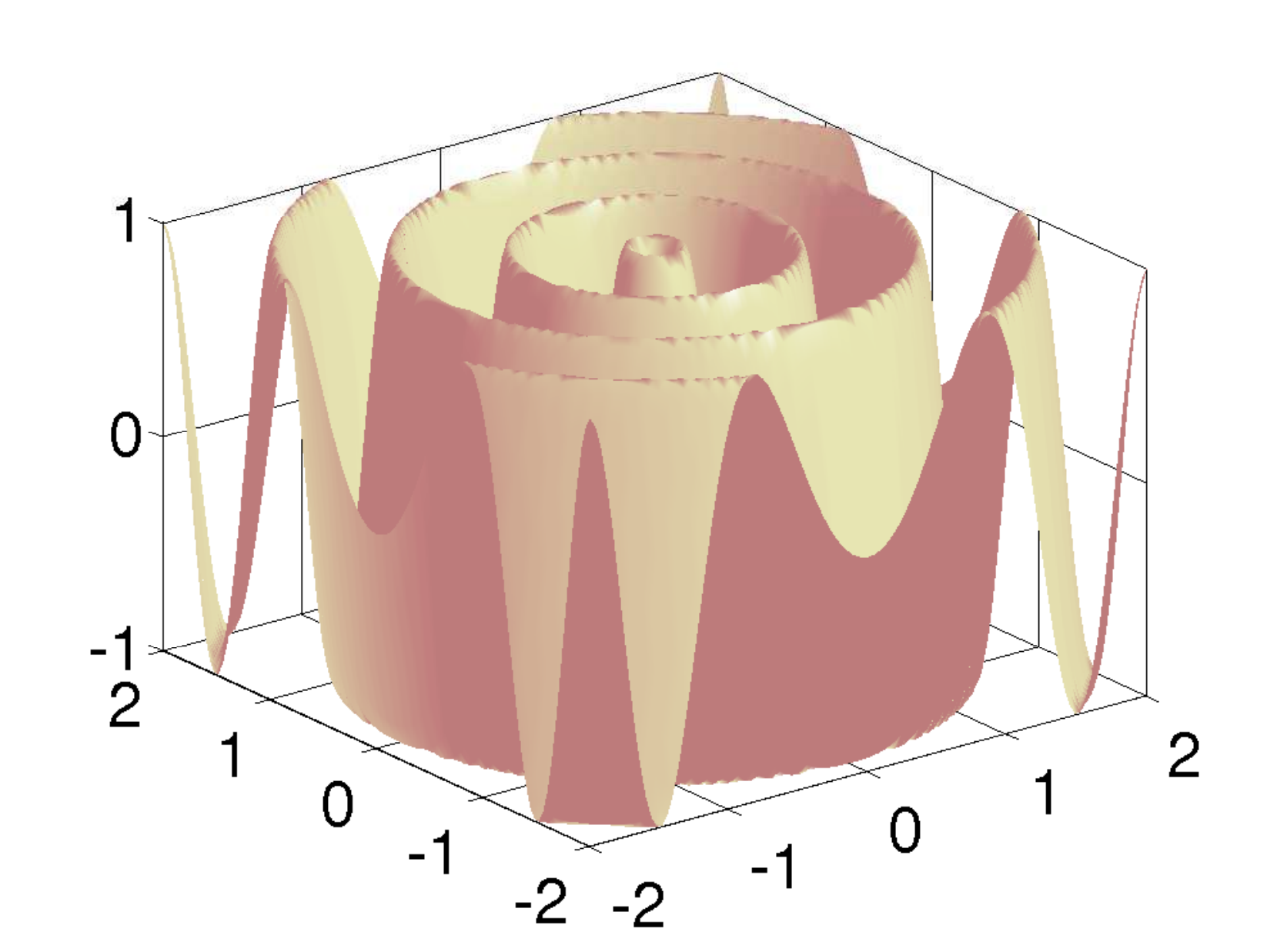}
\includegraphics[width=6cm]{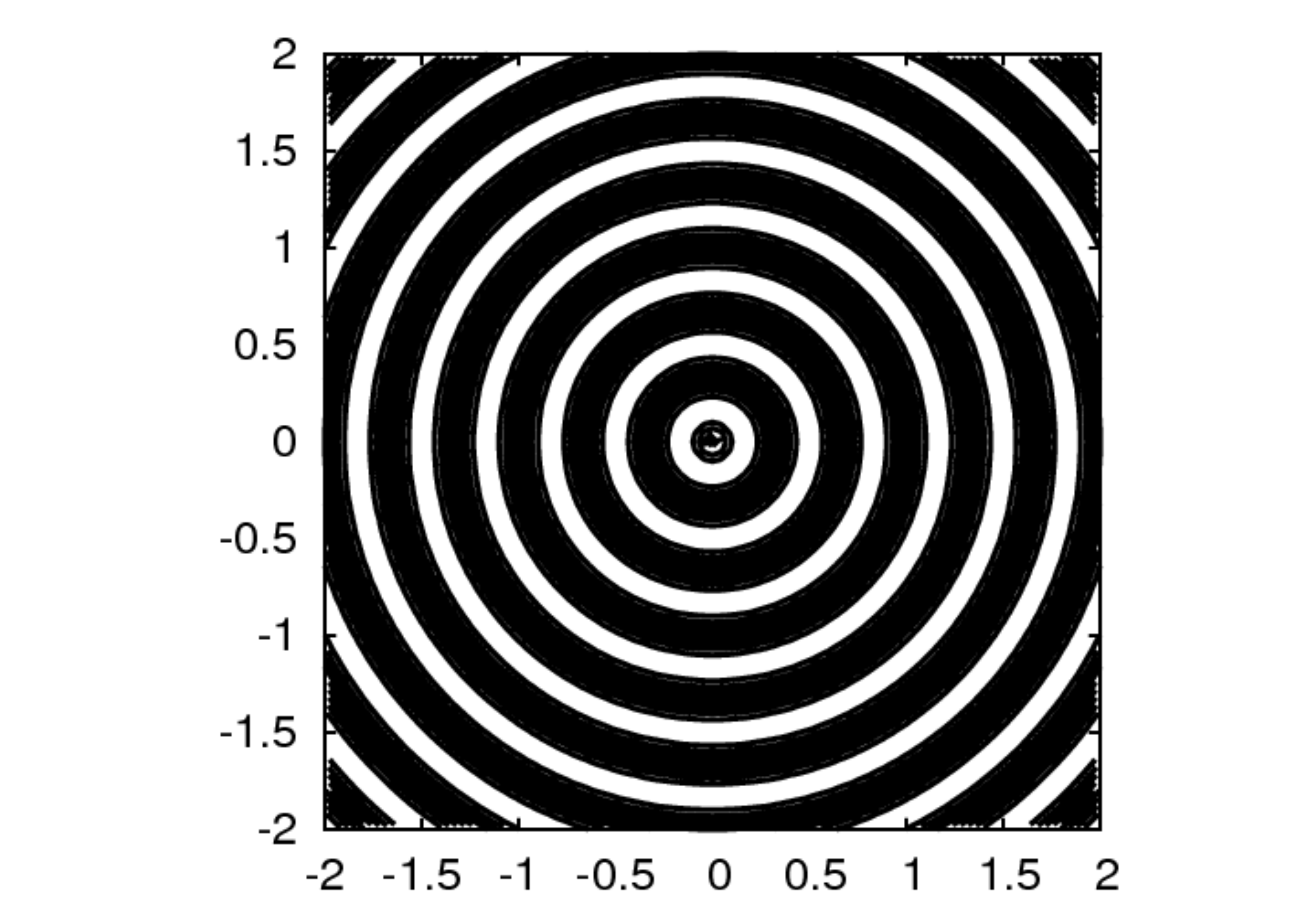}
\includegraphics[width=6cm]{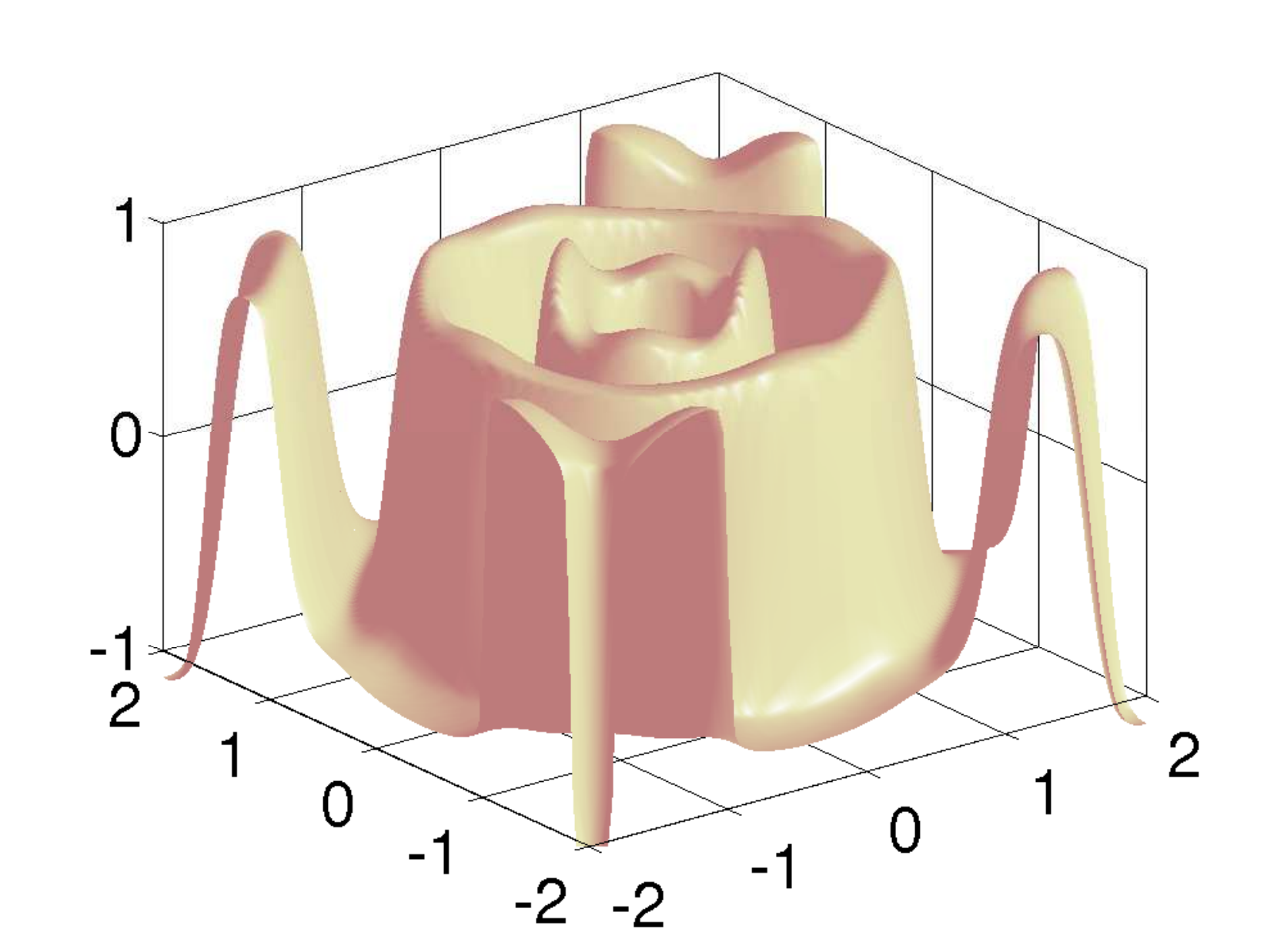}
\includegraphics[width=6cm]{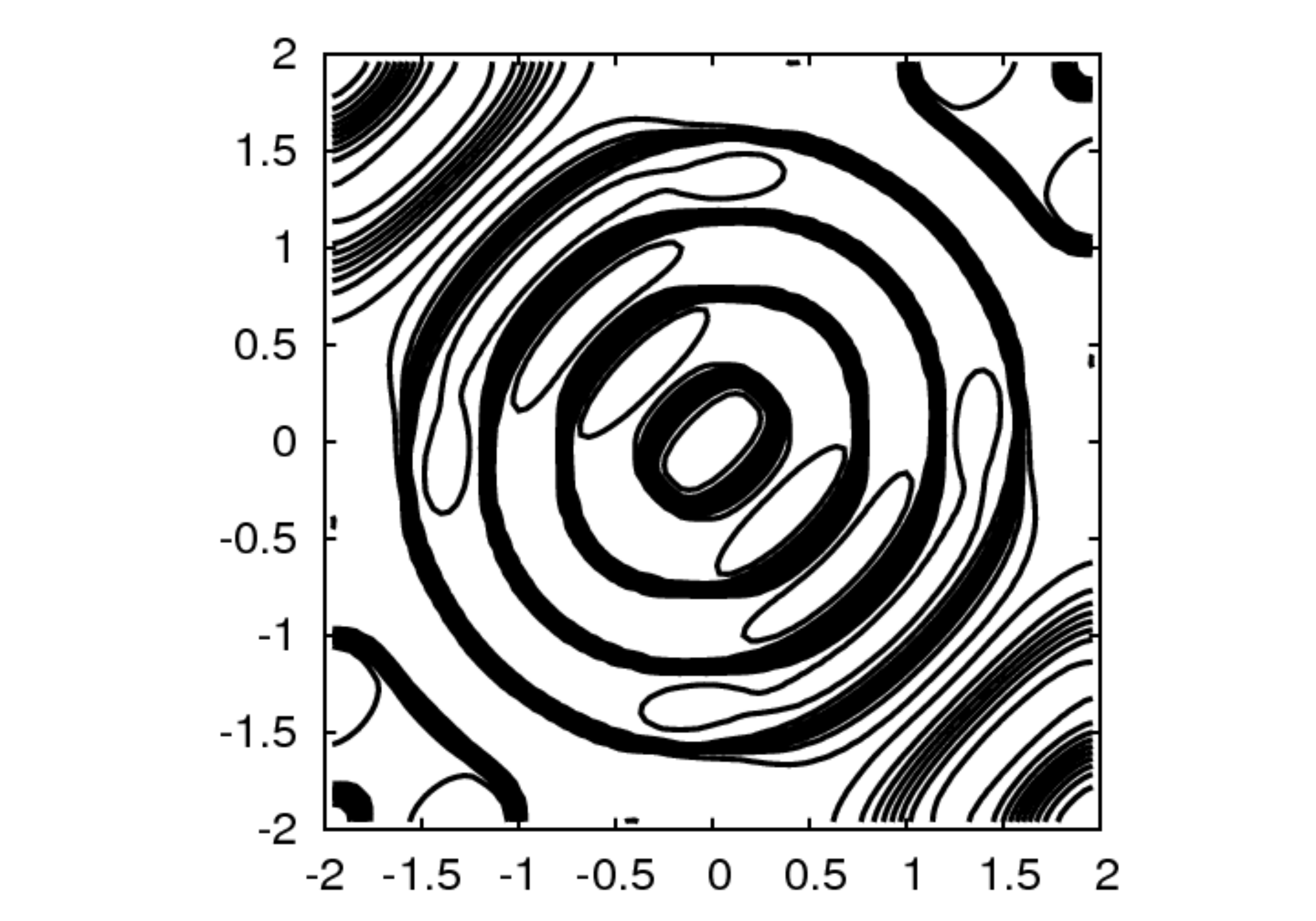}
\includegraphics[width=6cm]{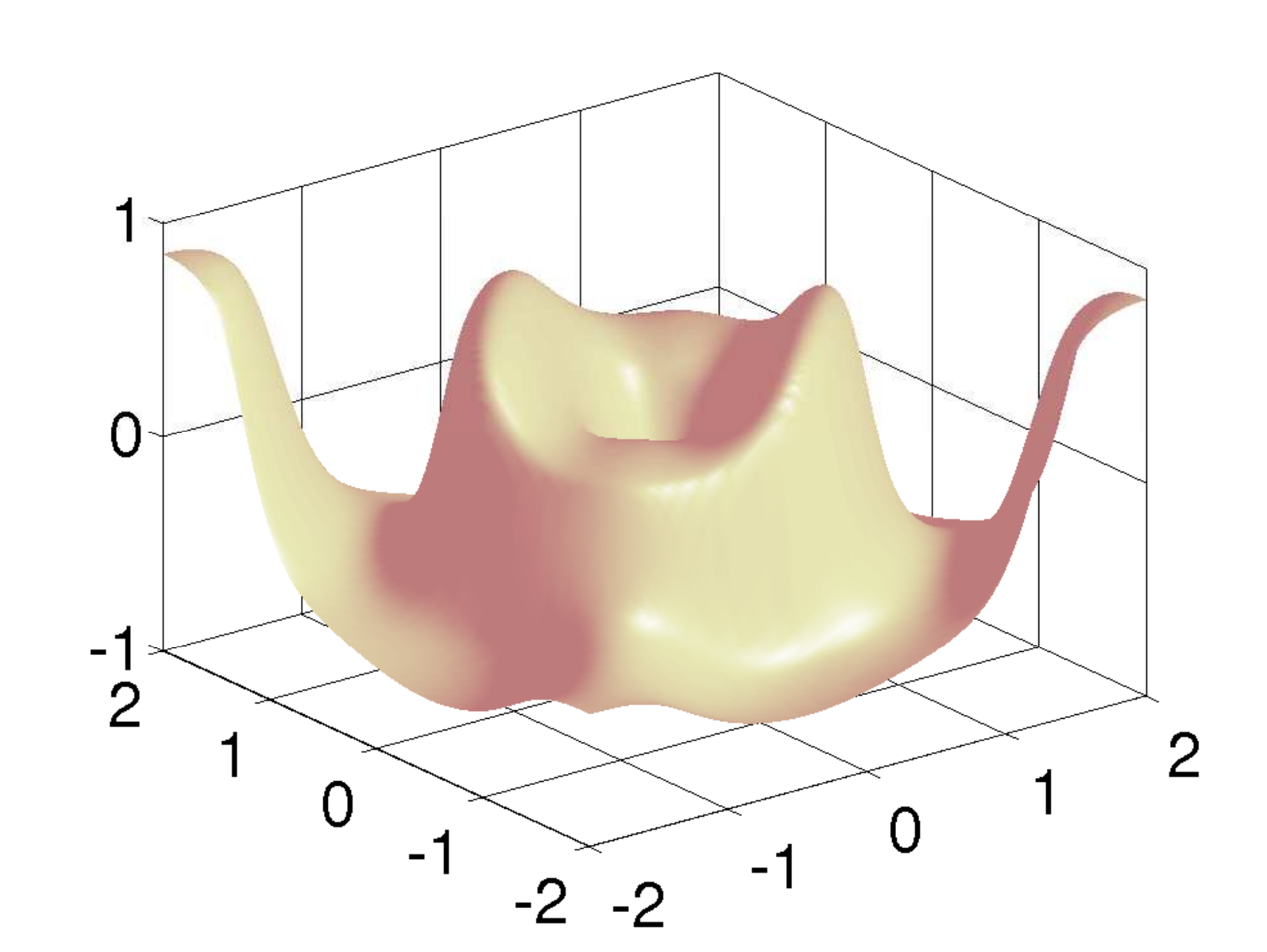}
\includegraphics[width=6cm]{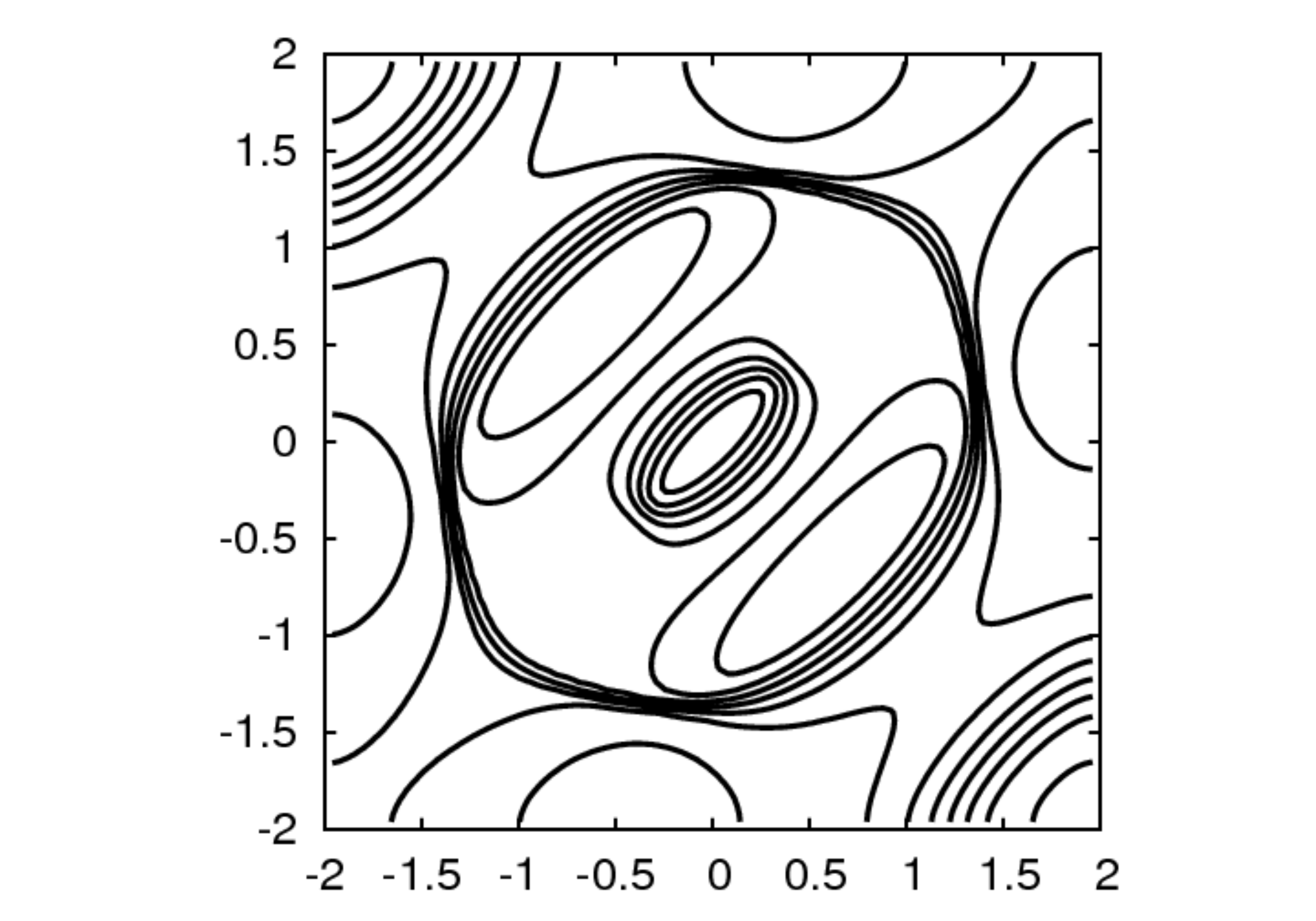}
\includegraphics[width=6cm]{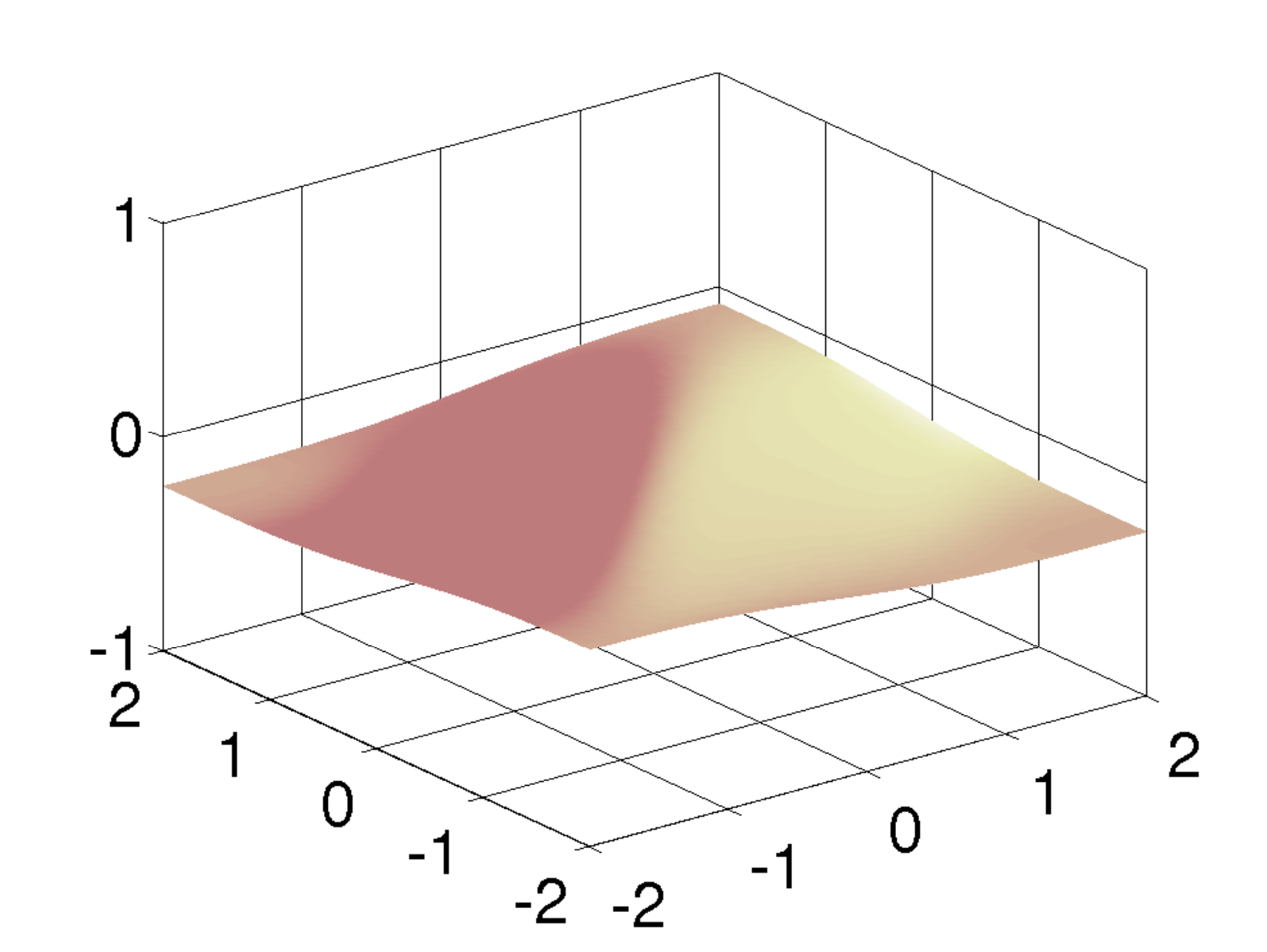}
\includegraphics[width=6cm]{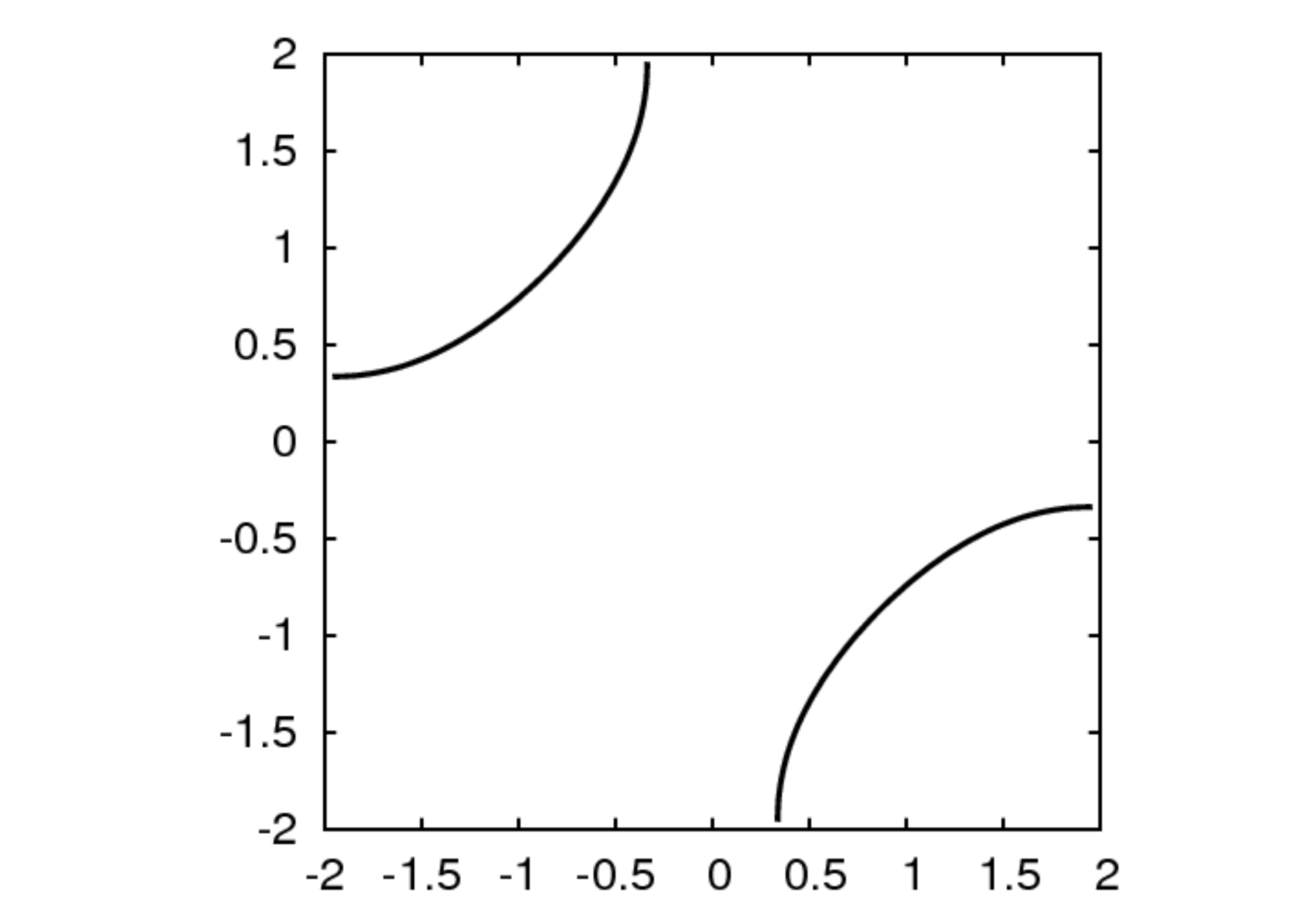}
}
\caption{Results obtained with the anisotropy given by (\ref{def:quadratic-form-anisotropy}) and (\ref{def:aniso-G-10-8-8-10}) at times $t=0$, $t=8 \cdot 10^{-6}$, $t=6.4 \cdot 10^{-5}$ and $t=1.024 \cdot 10^{-3}$ (graph of $u$ on the left, level-lines of $u$ on the right)}
\label{figure-2}
\end{figure}
And finally the Figure \ref{figure-3} reveals the behavior of the anisotropy given by (\ref{def:gamma2}) with $\epsilon_{abs}=0.001$ on the time interval $\left[0,0.006\right]$.
\begin{figure}
\center{
\includegraphics[width=6cm]{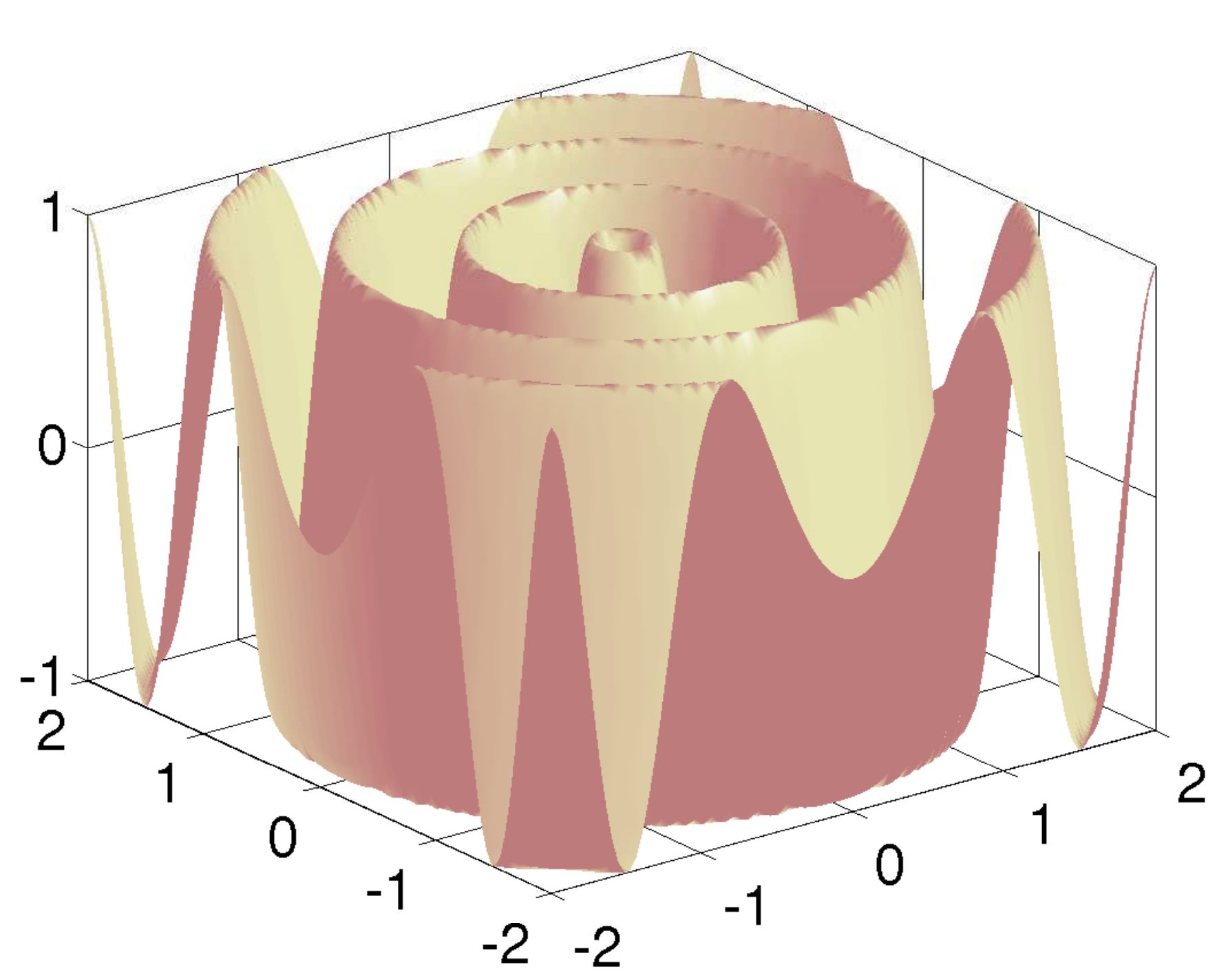}
\includegraphics[width=6cm]{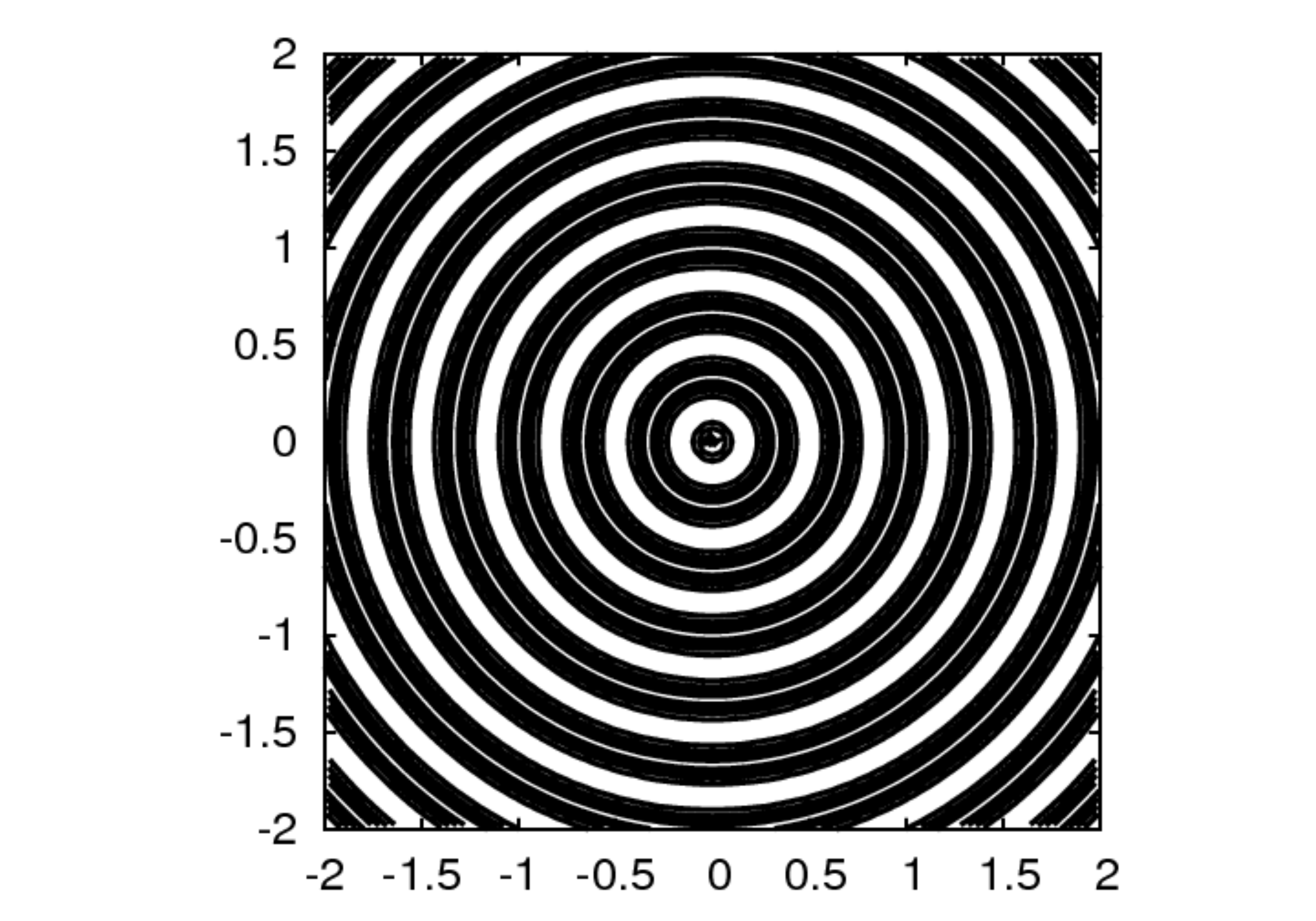}
\includegraphics[width=6cm]{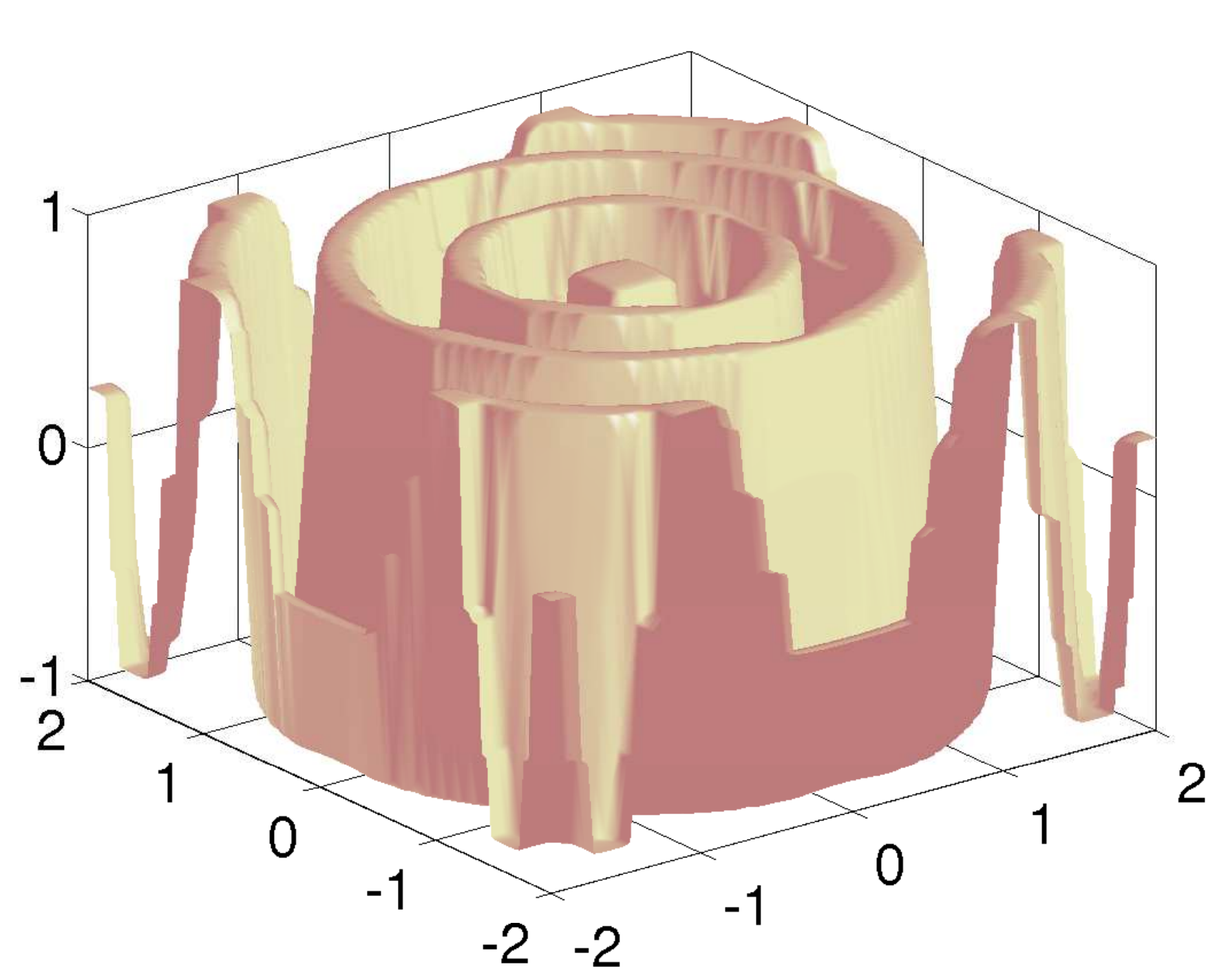}
\includegraphics[width=6cm]{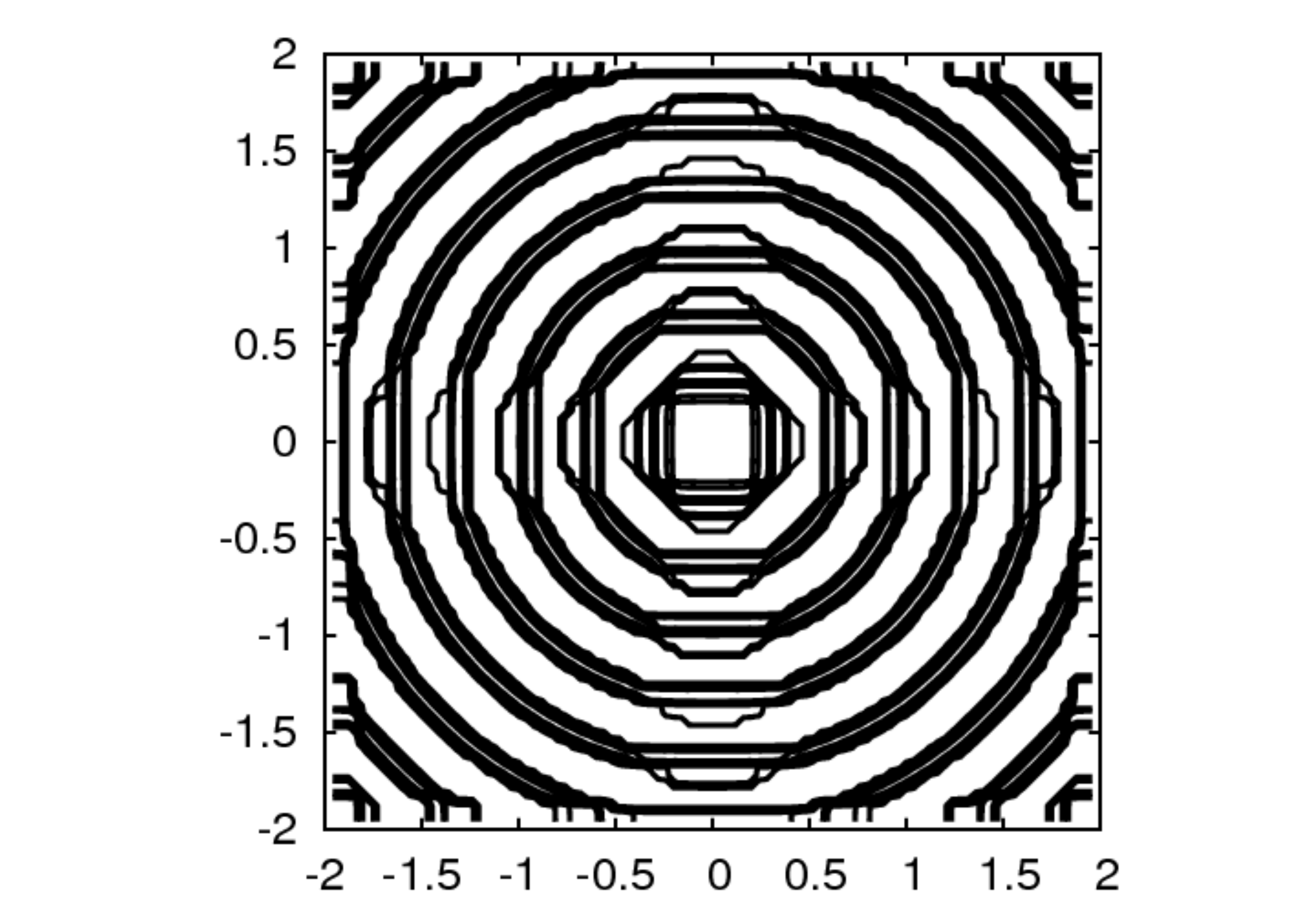}
\includegraphics[width=6cm]{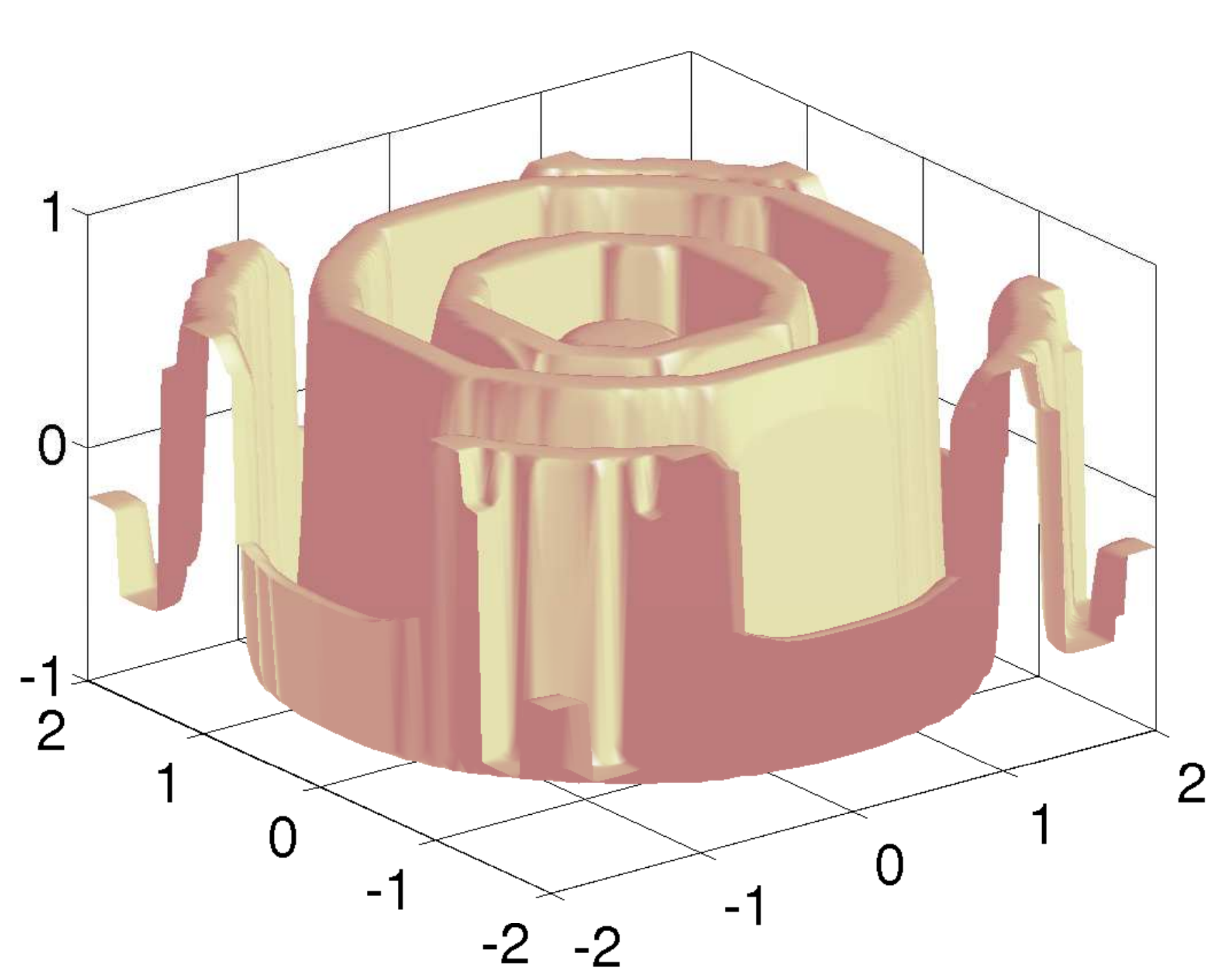}
\includegraphics[width=6cm]{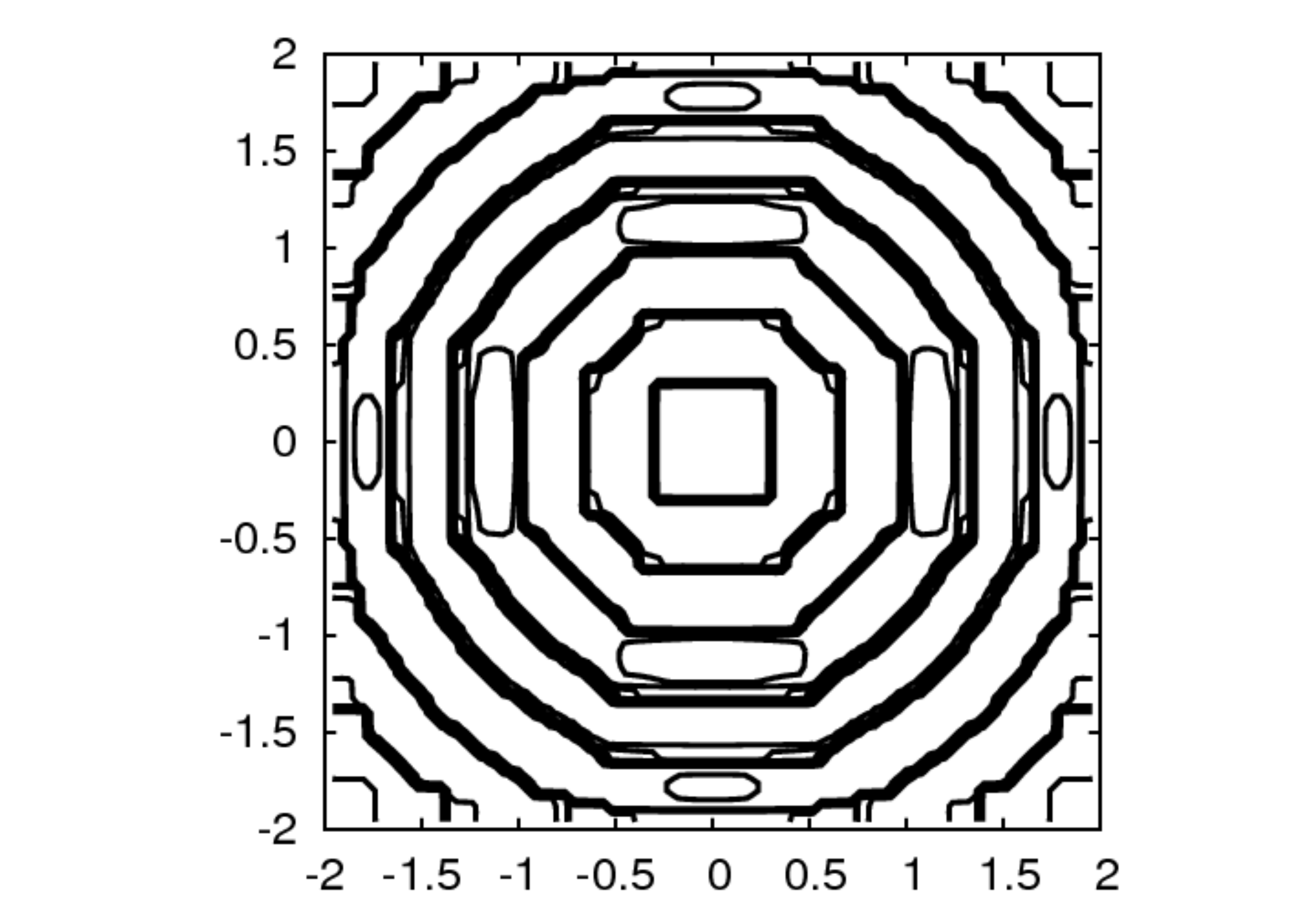}
\includegraphics[width=6cm]{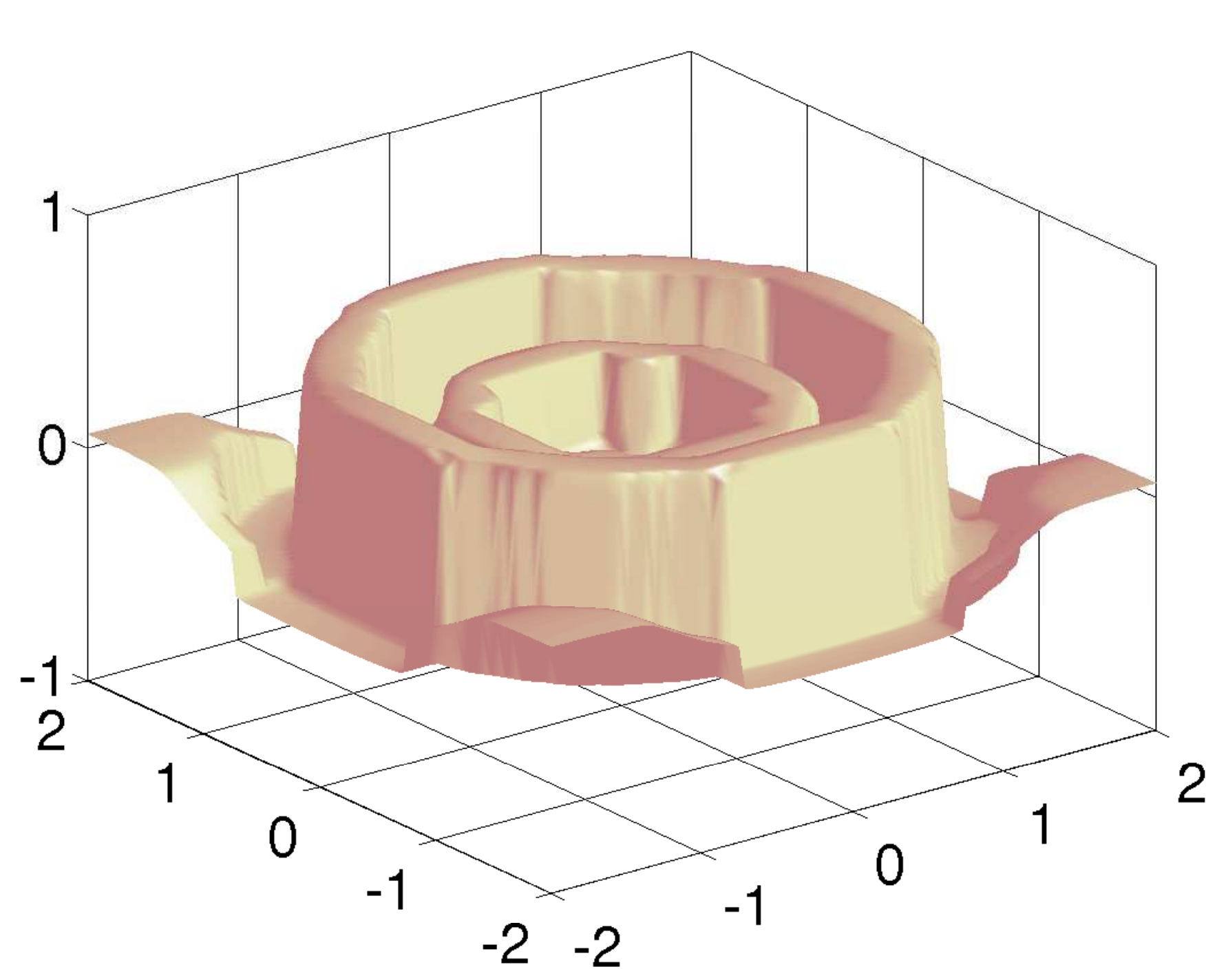}
\includegraphics[width=6cm]{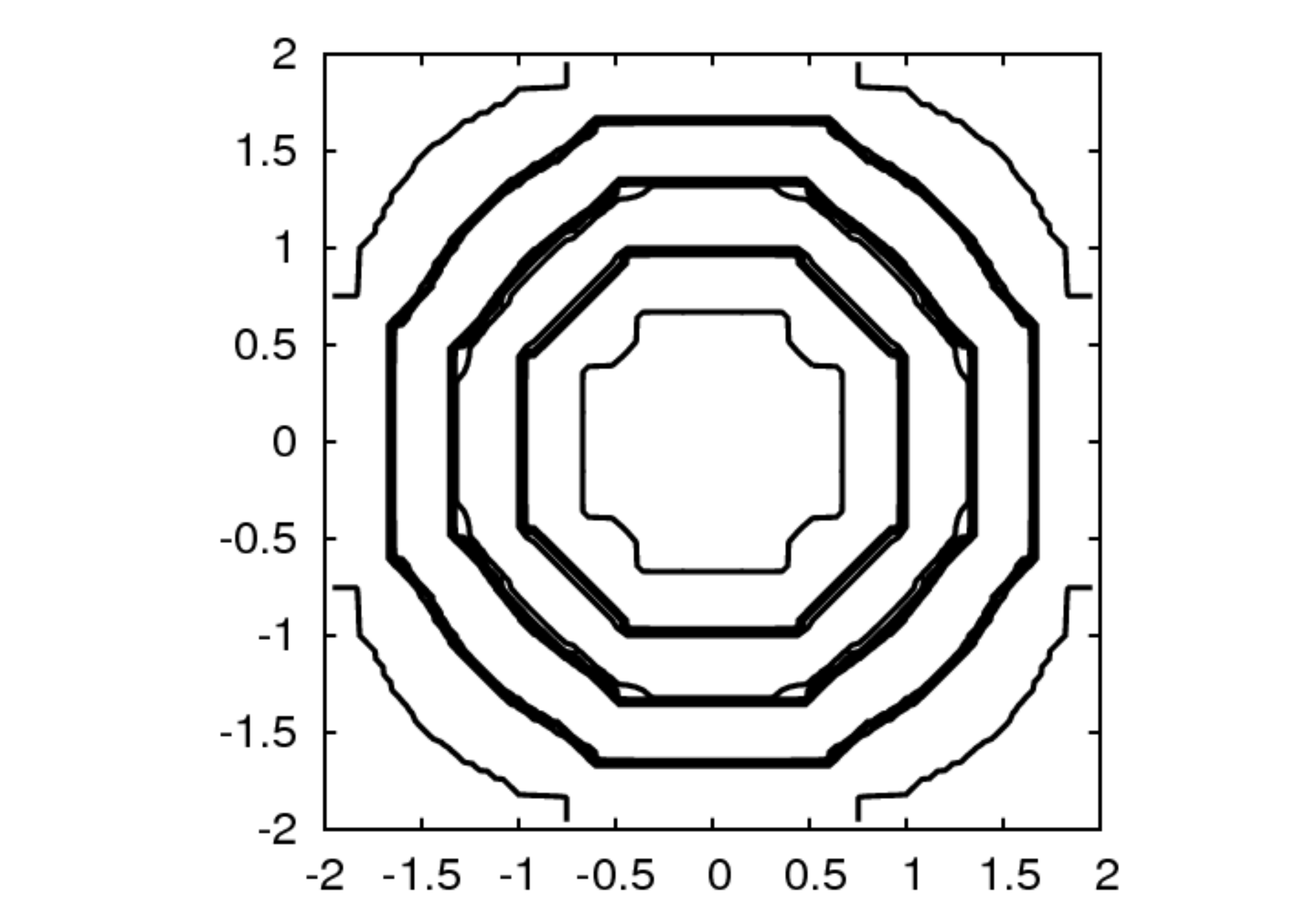}
}
\caption{Results obtained with the anisotropy given by (\ref{def:gamma2}) and $\epsilon_{abs}-0.1$ at times $t=0$, $t=\real{5}{-5}$, $t=0.001$ and $t=0.006$ -- not a steady state (graph of $u$ on the left, level-lines of $u$ on the right)}
\label{figure-3}
\end{figure}

\section{Conclusion}

We have presented new mathematical formulation for the anisotropic Willmore flow of graphs. We have proved energy equality for the new problem. To approximate the solution numerically we showed complementary finite volume based scheme. We also show how to re-formulate the numerical scheme by the finite difference method. This approach allowed us to prove stability of the numerical scheme. The computational part of the article contains experimental order of convergence evaluated on both studied anisotropies. We also showed few qualitative results.

\section{Acknowledgment}

This work was partially supported by the grant No. SGS11/161/OHK4/3T/14 of the Student Grant Agency of the Czech Technical University in Prague, Research Direction Project of the Ministry of Education of the Czech Republic No. MSM6840770010 and Research center of the Ministry of Education of the Czech Republic LC06052.

\clearpage


\begin{thebibliography}{10}

\bibitem{Benes-Mikula-Oberhuber-Sevcovic-2007-1}
M.~Bene\v{s}, K.~Mikula, T.~Oberhuber, and D.~\v{S}ev\v{c}ovi\v{c},
  \emph{Comparison study for level set and direct {L}agrangian methods for
  computing {W}illmore flow of closed plannar curves}, Computing and
  Visualization in Science \textbf{12} (2009), 307--317.

\bibitem{BertalmioCasellesHaroSapiro-2006}
M.~Bertalmio, V.~Caselles, G.~Haro, and G.~Sapiro, \emph{Handbook of
  mathematical models in computer vision}, ch.~PDE-Based Image and Surface
  Inpainting, pp.~33--61, Springer, 2006.

\bibitem{Canham-1970}
P.B. Canham, \emph{The minimum energy of bending as a possible explanation of
  biconcave shape of the human red blood cell}, J.Theoret.Biol. \textbf{26}
  (1970), 61--81.

\bibitem{Dec_Dzi_EEFTWFOG}
K.~Deckelnick and G.~Dziuk, \emph{Error estimates for the {W}illmore flow of
  graphs}, Interfaces Free Bound. \textbf{8} (2006), 21--46.

\bibitem{Dro_Rum_ALSFFWF}
M.~Droske and M.~Rumpf, \emph{A level set formulation for {W}illmore flow},
  Interfaces Free Bound. \textbf{6} (2004), no.~3, 361--378.

\bibitem{Du-Liu-Ryham-Wang-2005}
Q.~Du, C.~Liu, R.~Ryham, and X.~Wang, \emph{A phase field formulation of the
  {W}illmore problem}, Nonlinearity (2005), no.~18, 1249--1267.

\bibitem{Dzi_Kuw_Sch_EOECIRNEAC}
G.~Dziuk, E.~Kuwert, and R.~Sch\"atzle, \emph{Evolution of elastic curves in
  $\mathbbm{R}^n$: {E}xistence and computation}, SIAM J. Math. Anal.
  \textbf{41} (2003), no.~6, 2161--2179.

\bibitem{Oberhuber-2006}
T.~Oberhuber, \emph{Finite difference scheme for the {W}illmore flow of
  graphs}, Kybernetika \textbf{43} (2007), 855--867.

\bibitem{Oberhuber-2009}
\bysame, \emph{Complementary finite volume scheme for the anisotropic surface
  diffusion flow}, Proceedings of Algoritmy 2009 (A.~Handlovi\v{c}ov\'{a},
  P.~Frolkovi\v{c}, K.~Mikula, and D.~\v{S}ev\v{c}ovi\v{c}, eds.), 2009,
  pp.~153--164.

\bibitem{OberhuberSuzukiZabka-2011}
T.~Oberhuber, A.~Suzuki, and V.~{\v Z}abka, \emph{The cuda implentation of the
  method of lines for the curvature dependent flows}, Kybernetika \textbf{47}
  (2011), 251--272.

\bibitem{Willmore-2002}
T.~J. Willmore, \emph{Riemannian geometry}, Oxford University Press, 2002.

\bibitem{XuShu-2009}
Y.~Xu and W.~Shu, C., \emph{Local discontinuous galerkin method for surface
  diffusion and willmore flow of graphs}, Journal of Scientific Computing
  (2009), 375--390.

\end{thebibliography}

\providecommand{\bysame}{\leavevmode\hbox to3em{\hrulefill}\thinspace}
\providecommand{\MR}{\relax\ifhmode\unskip\space\fi MR }
\providecommand{\MRhref}[2]{%
  \href{http://www.ams.org/mathscinet-getitem?mr=#1}{#2}
}
\providecommand{\href}[2]{#2}

\end{document}